\documentclass[a4paper,12pt,]{article}
\usepackage{hyperref}
\usepackage{amsfonts}
\usepackage{amsmath}
\usepackage{amssymb}
\usepackage{amsthm}
\usepackage{enumerate}
\usepackage[mathscr]{eucal}
\usepackage{eqlist}
\usepackage{graphicx}
\usepackage{esint}
\usepackage{color}
\usepackage{booktabs} 
\usepackage{multirow} 

\usepackage[body={15.6true cm, 24.9true cm}]{geometry}

\newtheorem{theorem}{\hskip\parindent\bf Theorem}[section]
\newtheorem{lemma}{\hskip\parindent\bf Lemma}[section]
\newtheorem{proposition}{\bf Proposition}[section]
\newtheorem{remark}{\hskip\parindent\bf Remark}[section]
\newtheorem{definition}{\hskip\parindent\bf Definition}[section]
\numberwithin{equation}{section}
\numberwithin{equation}{section} \allowdisplaybreaks
\usepackage{mathdots}

\renewcommand\abstract{{\bf Abstract}}

\begin{document}

\title {\LARGE \bf Near Field Refraction Problem With Loss of Energy in Negative Refractive Index Material}

\author{{Feida Jiang$^{a,b}$,~~Haokun Sui$^{a,}$\thanks{Corresponding author.
E-mail address: suihaokun@seu.edu.cn}}}
\maketitle
\begin{center}
\begin{minipage}{12cm}
\begin{description}
\item \small
 $a$.~School of Mathematics and Shing-Tung Yau Center of Southeast University, Southeast University, Nanjing 211189, P.R. China
 \item \small
$b$.~Shanghai Institute for Mathematics and Interdisciplinary Sciences, Shanghai 200433, P.R. China
\end{description}
\end{minipage}
\end{center}


	

\begin{abstract}{\bf:}{\footnotesize
~This paper studies the near field refraction problem with loss of energy in negative refractive index material. Based on the relative refractive index $\kappa$, the analysis is categorized into two cases, namely $\kappa < -1$ and $-1 < \kappa < 0$. For each case, we give the definition of the refractor and discuss some crucial properties of it. The properties of Fresnel coefficients are also discussed. Based on these properties, the existence of the weak solution when the target measure is either discrete or a finite Radon measure are proved. Besides, the critical case $\kappa = -1$ is also discussed briefly at the end of this paper.}		
\end{abstract}
	
	

\tableofcontents

\section{Introduction}\label{Section 1}

\sloppy{}

Near field refraction problem refers that given two media \textsc{I} and \textsc{II}, and a light source in medium \textsc{I}, constructing a refracting surface $\mathcal{R}$ separating media \textsc{I} and \textsc{II}, such that all ray emitted from the light source refract through $\mathcal{R}$ to a specified point $P$ in medium \textsc{II}. This problem was first posed by Gutiérrez and Huang \cite{GH14} in 2014. In their work, an abstract framework called the continuous mapping and measure equation method was presented to prove the existence of the weak solution to the point source near field refraction problem. There are also some other methods to prove the existence of the weak solution to the near field refraction problem. In 2021, Liu and Wang \cite{LW21} showed that both  point light source and parallel light source near field refraction problem without loss of energy can be  formulated as a nonlinear optimisation problem, so that the theory of nonlinear optimization can be used to prove the existence of the weak solution. For the regularity of the near field refraction problem, in 2015, Gutiérrez and Tournier \cite{GT15} established the local $C^{1,\alpha}$ estimates of the solutions of parallel near field refractor and reflector problem, thereby obtaining the local regularity of the corresponding weak solutions. In 2021, Gutiérrez and Tournier \cite{GT21} advanced the regularity theory for the point source near field refraction problem by establishing local $C^{1,\alpha}$ estimates. In 2021, a numerical scheme to solve the point source near field refractor problem to arbitrary precision was proposed by Gutiérrez and Mawi \cite{GM21}. The convergence and universality of the algorithm were also demonstrated in their work. Notice that these results are all obtained in the media with positive refractive index, and do not consider the loss of energy caused by the reflection. For other relevant studies related to the near field refraction problem, see \cite{Ji23}\cite{JT14}\cite{JT18}\cite{Tr21}.

In fact, when the light ray emitted from medium \textsc{I} strikes the interface between medium \textsc{I} and \textsc{II}, it gives two rays, some of the ray will be refracted into medium \textsc{II}, while the other ray will be reflected back to medium \textsc{I}. Therefore, the energy of the incident light ray is not equal to that of the refracted light ray. The distribution of energy between reflected and refracted light ray can be described by reflection and transmission coefficients. In 2013, Mawi and Gutiérrez \cite{GM13} conducted the first mathematical investigation of the far field refraction problem with loss of energy in positive refractive index material, proving the existence of the weak solution and deriving the associated Monge-Amp\`ere type equation. In 2016, Stachura proved the existence of the weak solution to the near field refraction problem with loss of energy in positive refractive index media in \cite{St16}.

Negative refraction was first postulated by the Russian scientist Veselago in 1968 \cite{Ve68}. This phenomenon describes a unique refractive behavior of light. When a light wave strikes the boundary between a positive and a negative refractive index material, it deviates from conventional refraction. Specifically, the incident and refracted waves lie on the same side of the interface normal. However, this theoretical prediction lacked experimental validation for nearly half a century. It was not until 2000 that Smith et al. \cite{SP00} successfully fabricated the world's inaugural medium with both negative equivalent permittivity and permeability in the microwave region using copper-based composites, marking the transition of negative refraction from theoretical conjecture to practical realization. Building on this breakthrough, Shelby et al. \cite{SS01} constructed a prism from existing negative refractive index materials in 2001, directly validating the existence of negative refraction through experiments and demonstrating that light incident on the surface of a negative refractive index medium refracts to the same side of the interface normal as the incident wave. Since the dawn of the 21st century, negative refractive index materials have found extensive applications in fields such as optical invisibility \cite{CT05}, perfect lens imaging \cite{PR02}\cite{Pen00}, wireless directional radiation \cite{ET02}, and manufacturing of novel optical devices such as high-capacity optical discs \cite{LH04}.

Mathematically, near field and far field refraction problem without loss of energy in negative refractive index material was first studied by Gutiérrez and Stachura in 2015 \cite{GS15}. Later in 2016, Gutiérrez and Stachura \cite{GS16} further generalized their previous work by investigating double refraction in both near and far fields in negative refractive index media. In 2017, Stachura \cite{St17} proved the existence of the weak solution to the near field and far field refraction problem without loss of energy in negative refractive index material by using Minkowski method and the optimal transmission method respectively. Actually,  refraction with loss of energy also occurs in negative refractive index material \cite{DC07}. Figure \ref{fig1} shows the refraction problem with loss of energy in negative refractive index material, indicating that when an incident light ray having direction of propagation $x\in S^{n-1}$ strikes at $\mathcal{R}$, it will split into two rays: a reflected ray in direction $r\in S^{n-1}$ back into medium \textsc{I} and a refracted ray in direction $m\in S^{n-1}$ transmitted into medium \textsc{II}. This paper mainly focuses on the existence of the weak solution. The results of the existence of the weak solution for the point source near field refraction are shown in the following Table \ref{tab1}. 

\begin{figure}[h]
  \centering
  \includegraphics[width=8.3cm]{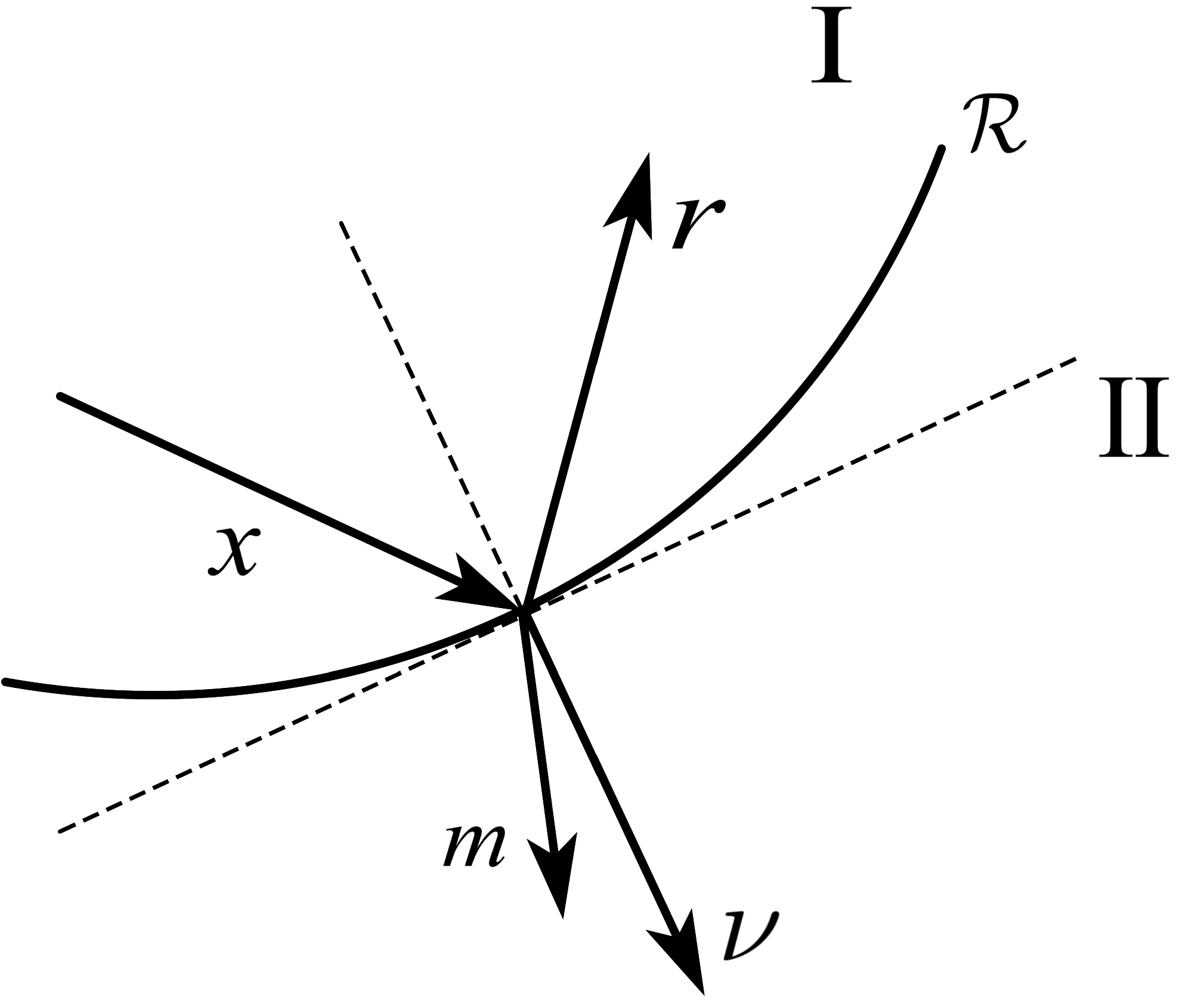}
\caption{Sketch of the refraction problem with loss of energy in negative refractive index material.}\label{fig1}
\end{figure}

\begin{table}[htbp]
  \centering
  \caption{Results of the existence of the weak solution for the point source near field refraction}
  \label{tab1}
  \begin{tabular}{ccc}
    \toprule
    $\kappa$ & without loss of energy & with loss of energy \\
    \midrule
    $\kappa > 1$       & Theorem 6.9 in \cite{GH14}   & Remark \ref{rem1.1}     \\
    $0< \kappa < 1$    & Theorem 5.8 in \cite{GH14}   & Theorem 7.4 in \cite{St16}     \\
    $\kappa < -1$      & Theorem 3.6 in \cite{St17}   & Theorem \ref{thm1.1}    \\
    $\kappa = -1$      & Section \ref{Section 6.3}    & Not occur    \\
    $-1< \kappa < 0$   & Theorem 2.9 in \cite{St17}   & Theorem \ref{thm1.2}    \\
    \bottomrule
  \end{tabular}
\end{table}

\begin{remark}
  The existence of the weak solution to the point source near field refraction problem with loss of energy for the case $\kappa > 1$ can be proved by using the similar method as the case $\kappa < -1$. This paper mainly focuses on the near field refraction problem with loss of energy in negative refraction index material. Therefore, we do not show the existence of the weak solution to the point source near field refraction problem with loss of energy for the case $\kappa > 1$ in this paper.
  \label{rem1.1}   
\end{remark}

In this paper, we consider the following problem: Suppose that $\Omega$ is a domain in $S^{n-1}$ and $D$ is a set contained in a $n-1$  dimensional hypersurface, $f$ and $g$ are two integrable functions on $\Omega$ and $D$ respectively, that is, $f\in L^{1}(\Omega)$, $g\in L^{1}(D)$. Consider two homogeneous, isotropic media: medium \textsc{I} and medium \textsc{II}, surrounded by  $\Omega$ and $D$ respectively which have different optical densities. Given a point $P\in D$, we want to construct a surface $\mathcal{R}$ separating media \textsc{I} and \textsc{II}, such that all rays emitted from the original point $O\in \Omega$ in medium \textsc{I} with intensity $f(x)$, $x\in \Omega$, are refracted by $\mathcal{R}$ to a specific point $P \in D$ in medium \textsc{II} and the corresponding illumination intensity is $g(P)$. In this paper, we first assume that the target measure $\mu$ equals the finite sum of $\delta$-measure, then we use discrete measures to approximate general Radon measure. Assuming that the refractive index of medium \textsc{I} is $n_{1}>0$, the refractive index of medium \textsc{II} is $n_{2}<0$, and set the relative refractive index $\kappa= \dfrac{n_{2}}{n_{1}}$, so $\kappa<0$. Notice that in application, it is natural to study the refraction problem in $S^{2}$, see Figure \ref{fig2}. However, in this paper, we directly study the problem in $S^{n-1}$ with $n \geq 2$ for its generality.

\begin{figure}[h]
  \centering
  \includegraphics[width=8.3cm]{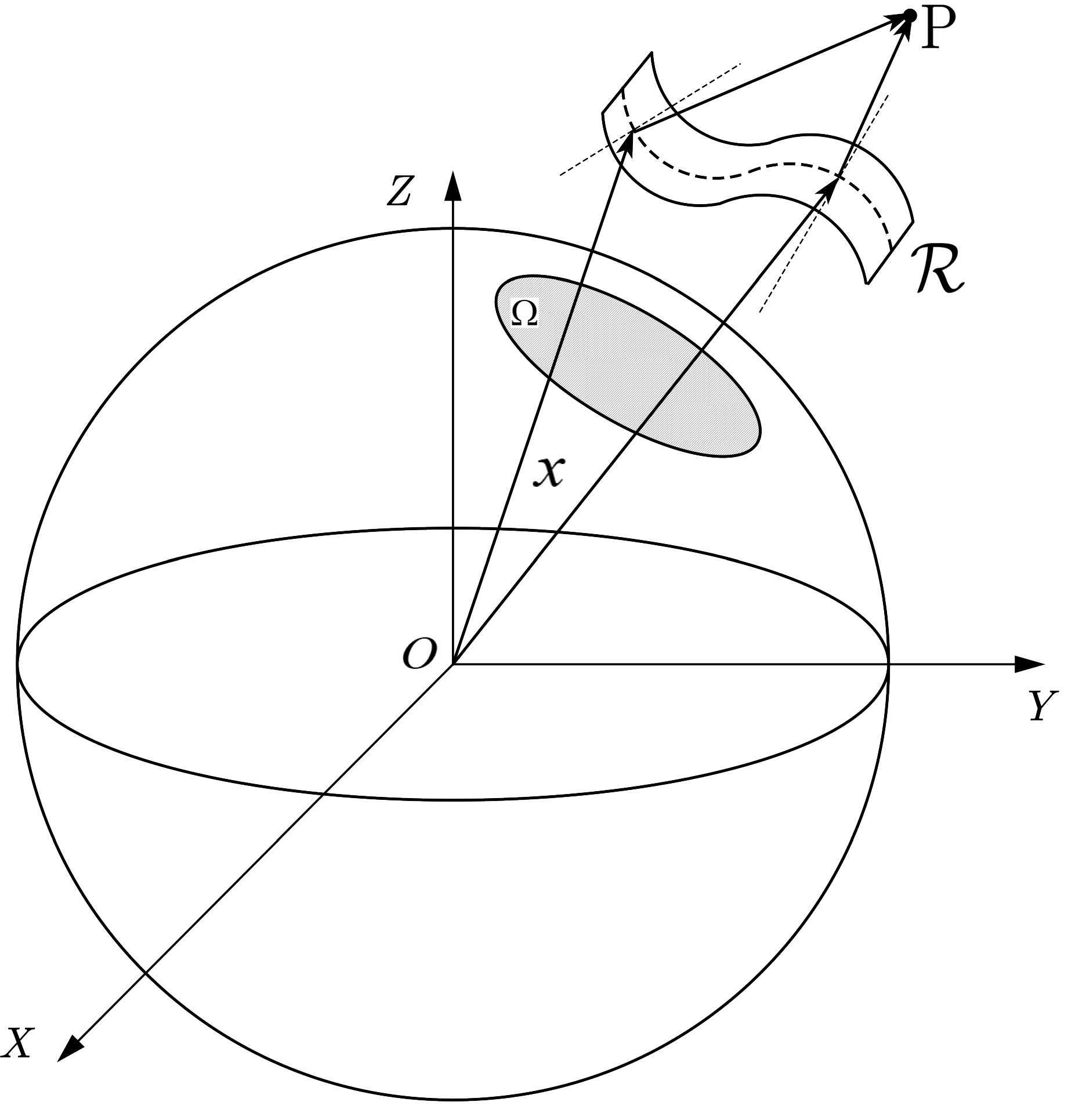}
\caption{Statement of the near field refraction problem in $S^{2}$.}\label{fig2}
\end{figure}

In order to solve this problem, we first need some assumptions.

For the case $\kappa < -1$:
\begin{itemize}
  \item[\bf{(A0\text{-}1)}] There exists $\tau \in (0,1 - \dfrac{1}{\kappa})$, such that $\inf\limits_{x \in \bar{\Omega}, P \in \bar{D}} x \cdot \dfrac{P}{\vert P \vert} \geq \tau + \dfrac{1}{\kappa}$.
  \item[\bf{(A0\text{-}2)}] For any $m \in S^{n-1}$ and $x \in \mathcal{C}_{r_{0}}$, $\bar{D} \cap \{x + tm; t \geq 0\}$ contains at most one point, where $r_{0} \in \left(0,\dfrac{\tau^{2}\kappa^{2}}{(1 + \sqrt{2})^{2}(1 - \kappa)^{2}}\inf\limits_{P \in \bar{D}}\vert P \vert\right)$ and $\mathcal{C}_{r_{0}} = \{tx; x \in \bar{\Omega}, 0 < t \leq r_{0}\}$ is a cone.
  \item[\bf{(A1)}] $f \in L^{1}(\bar{\Omega})$ with $\inf\limits_{\bar{\Omega}}f > 0$.
  \item[\bf{(A2)}] $P_{1},P_{2},\ldots, P_{m},~m \geq 2$ are discrete points in $\bar{D}$, $g_{1}, g_{2}, \ldots, g_{m}>0$.
  \item[\bf{(A3)}] For $j = 2,\ldots,m$, the distance between $P_{1}$ and $P_{j}$ is sufficient large.
  \item[\bf{(A4)}] There exists $\varepsilon > 0$, such that $\inf\limits_{x\in\bar{\Omega},1 \leq j \leq m}x \cdot \dfrac{P_{j}}{\vert P_{j} \vert} \geq \dfrac{1}{\kappa} + \varepsilon$.
  \item[\bf{(A5)}] There holds $\displaystyle\int_{\bar{\Omega}}f(x)dx \geq \dfrac{1}{1 - C_{\varepsilon}}\mu(\bar{D})$, where $C_{\varepsilon}$ is the upper bound of reflection coefficient.
\end{itemize}
For the case $\kappa < -1$, $(A0\text{-}1)$ and $(A0\text{-}2)$ are the conditions to ensure that refraction can occur, so we always assume that $(A0\text{-}1)$ and $(A0\text{-}2)$ hold in the context. In what follows, these conditions will be implicitly assumed in all lemmas and theorems.

For the case $-1 < \kappa < 0$:
\begin{itemize}
  \item[\bf{(B0\text{-}1)}] There exists $\tau \in (0,1 + \kappa)$, such that $x \cdot P \geq (\tau - \kappa) \vert P \vert$ holds for all $x \in \bar{\Omega}$ and $P \in \bar{D}$.
  \item[\bf{(B0\text{-}2)}] For any $m \in S^{n-1}$ and $x \in \mathcal{C}_{r_{0}}$, $\bar{D} \cap \{x + tm; t \geq 0\}$ contains at most one point, where $r_{0} \in \left(0,\dfrac{\tau}{1 - \kappa}dist(0,\bar{D})\right)$ and $\mathcal{C}_{r_{0}} = \{tx; x \in \bar{\Omega}, 0 < t \leq r_{0}\}$ is a cone.
  \item[\bf{(B1)}] $f \in L^{1}(\bar{\Omega})$ with $\inf\limits_{\bar{\Omega}}f > 0$.
  \item[\bf{(B2)}] $P_{1},P_{2},\ldots, P_{m},~m \geq 2$ are discrete points in $\bar{D}$, $g_{1}, g_{2}, \ldots, g_{m}>0$.
  \item[\bf{(B3)}] For $j = 2,\ldots,m$, the distance between $P_{1}$ and $P_{j}$ is sufficient large.
  \item[\bf{(B4)}] There exists $\varepsilon > 0$, such that $\inf\limits_{x\in\bar{\Omega},1 \leq j \leq m}x \cdot \dfrac{P_{j}}{\vert P_{j} \vert} \geq \kappa + \varepsilon$.
  \item[\bf{(B5)}] There holds $\displaystyle\int_{\bar{\Omega}}f(x)dx \geq \dfrac{1}{1 - C_{\varepsilon}}\mu(\bar{D})$, where $C_{\varepsilon}$ is the upper bound of reflection coefficient.
\end{itemize}
For the case $-1 < \kappa < 0$, $(B0\text{-}1)$ and $(B0\text{-}2)$ are the conditions to ensure that refraction can occur, so we always assume that $(B0\text{-}1)$ and $(B0\text{-}2)$ hold in the context. In what follows, these conditions will be implicitly assumed in all lemmas and theorems.

The symbols and terminologies in the assumptions will be explained in the following sections.

Now we state the main results of this paper.

\begin{theorem}
  Suppose that the assumptions $(A1)$ $-$ $(A5)$ hold, then the weak solution to the near field refraction problem with loss of energy with emitting illumination intensity $f(x)$ and prescribe refracted illumination intensity $\mu$ for the case $\kappa < -1$ exists, where $\mu$ is general Radon measure.
  \label{thm1.1}
\end{theorem}

\begin{theorem}
  Suppose that the assumptions $(B1)$ $-$ $(B5)$ hold, then the weak solution to the near field refraction problem with loss of energy with emitting illumination intensity $f(x)$ and prescribe refracted illumination intensity $\mu$ for the case $-1< \kappa < 0$ exists, where $\mu$ is general Radon measure.
  \label{thm1.2}
\end{theorem}

In order to prove Theorems \ref{thm1.1} and \ref{thm1.2}, we first give Snell law in vector form and discuss the physical constraints of $\bar{\Omega}$ and $\bar{D}$ to ensure that total internal reflection cannot occur. Then we introduce the refractor and give some properties for the case $\kappa < -1$ and $-1 < \kappa < 0$ respectively. Next, we give Fresnel formula in negative refractive index material and discuss some properties of Fresnel coefficients. At last, the existence results when the underlying measure is the finite sum of $\delta$-measures is proved by using approximation by ovals, see Theorems \ref{thm6.1} and \ref{thm6.2}. Using these results, we can prove the main results of this paper. Additionally, as a critical case, the near field refraction problem for the case $\kappa = -1$ will be briefly mentioned in Section \ref{Section 6.3}. 

The rest of the content is organized as follows: In Section \ref{Section 2}, we give the Snell law in vector form in negative refractive index material, which is the preliminary of this paper. In Section \ref{Section 3}, we give the definition of the refractor for the case $\kappa < -1$ and $-1 < \kappa < 0$ and discuss some properties. The Fresnel coefficients and their properties are given in Section \ref{Section 4}. In Section \ref{Section 5}, we give the definition of the weak solution to the near field refraction problem with loss of energy in negative refractive index medium, and the existence of the weak solution in both in discrete and general situations are proved in Section \ref{Section 6}. Finally in Section \ref{Section 7}, we summarize our work and give some open problems in near field refraction problem with loss of energy.

\section{Snell law in vector form}\label{Section 2}

\sloppy{}

In this section, we briefly introduce Snell law in vector form in negative refractive index material.

Suppose $\mathcal{R}$ is a surface in $\mathbb{R}^{n}$ that separates two homogeneous and isotropic media \textsc{I} and \textsc{II}, with refractive indices $n_{1}>0$ and $n_{2}<0$. A ray of light emitted from $O\in S^{n-1}$ in medium \textsc{I} with direction $x\in S^{n-1}$ strikes at $\Gamma$ at the point $Q$, then the refracted ray has the direction $m\in S^{n-1}$ in medium \textsc{II}. Let $\nu$ be the unit normal to $\Gamma$ at $Q$ going towards medium \textsc{II}, $\theta_{1}$ denote the angle between $x$ and $\nu$, referred to as the angle of incidence; similarly, $\theta_{2}$ denote the angle between 
$m$ and $\nu$, defined as the angle of refraction. Then we have the well-known Snell law in scalar form:
\begin{equation}\label{2.1}
  n_{1}\sin\theta_{1} = n_{2}\sin\theta_{2}.
\end{equation}
This law can be written in vector form as:
\begin{equation}\label{2.2}
  n_{1}(x\times\nu) = n_{2}(m\times\nu).
\end{equation}

From (\ref{2.2}), it is easily seen that $x,m$ and $\nu$ are in the same plane. Since we have set $\kappa=\dfrac{n_{2}}{n_{1}}$, (\ref{2.2}) can be written as
\begin{equation}\label{2.3}
  x-\kappa m = \lambda \nu,
\end{equation}
where $\lambda\in \mathbb{R}$ is given by
\begin{equation}\label{2.4}
  \lambda = x\cdot \nu + \sqrt{(x\cdot \nu)^{2} - (1 - \kappa^{2})} = x\cdot \nu + \vert \kappa \vert \sqrt{1 - \kappa^{-2}(1 - (x\cdot \nu)^{2})}.
\end{equation}
Defining
\begin{equation}\label{2.5}
  \Phi(t) = t + \vert \kappa \vert \sqrt{1 - \kappa^{-2}(1 - t^{2})},
\end{equation}
then 
$$\lambda = \Phi (x \cdot \nu)$$.

We summarize the physical constraints of refraction in the following lemma whose proof is in \cite{GS15}.

\begin{lemma}
  Suppose that the refractive indices of media \textsc{I} and \textsc{II} are given by $n_{1}>0$ and $n_{2}<0$, and set $\kappa = \dfrac{n_{2}}{n_{1}}$.
    
    (a) If $\kappa <-1$, a light ray in medium \textsc{I} in the direction $x \in S^{n-1}$ is refracted by some surface into a light ray in medium \textsc{II} in the direction $m \in S^{n-1}$ if and only if $x\cdot m \geq \dfrac{1}{\kappa}$.
   
    (b) If $-1<\kappa <0$, a light ray in medium \textsc{I} in the direction $x \in S^{n-1}$ is refracted by some surface into a light ray in medium \textsc{II} in the direction $m \in S^{n-1}$ if and only if $x\cdot m \geq \kappa$.
\label{lem2.1}
\end{lemma}

\begin{remark}
 Lemma \ref{lem2.1} is typically used to solve the refraction problem without loss of energy. However, since this paper considers the refraction problem with loss of energy, we need to strengthen the conclusions of Lemma \ref{lem2.1} appropriately, see the following:
 
  Suppose that the refractive indices of media \textsc{I} and \textsc{II} surrounded by  $\Omega$ and $\Omega^{\ast}$ respectively, are given by $n_{1}>0$ and $n_{2}<0$, and set $\kappa = \dfrac{n_{2}}{n_{1}}$.
   
   (a) For the case $\kappa <-1$, we assume that there exists $\varepsilon >0$, such that
   \begin{equation}\label{2.6}
     \inf_{x \in \bar{\Omega}, m \in \bar{\Omega}^{\ast}} x \cdot m \geq \frac{1}{\kappa} + \varepsilon,
   \end{equation}
   then from Lemma \ref{lem2.1} (a), a light ray in medium \textsc{I} in the direction $x \in S^{n-1}$ can be refracted by some surface into a light ray in medium \textsc{II} in the direction $m \in S^{n-1}$.
   
   (b) For the case $-1<\kappa <0$, we assume that there exists $\varepsilon >0$, such that
   \begin{equation}\label{2.7}
     \inf_{x \in \bar{\Omega}, m \in \bar{\Omega}^{\ast}} x \cdot m \geq \kappa + \varepsilon,
   \end{equation}
   then from Lemma \ref{lem2.1} (b), a light ray in medium \textsc{I} in the direction $x \in S^{n-1}$ can be refracted by some surface into a light ray in medium \textsc{II} in the direction $m \in S^{n-1}$.
  \label{rem2.1}
\end{remark}

\section{Refractor and its properties}\label{Section 3}

\sloppy{}

In this section, we give the definition of refractor for the case $\kappa < -1$ and $-1 < \kappa < 0$ which stems from \cite{St17}. Then we discuss some properties of the refractor.

\subsection{Case $\kappa < -1$}\label{Section 3.1}

We first give some conditions for $\Omega$ and $D$.

Let $\Omega \subset S^{n-1}$ be a domain, satisfies $\vert \partial \Omega \vert = 0$. Let $D \subset \mathbb{R}^{n}$ be the target domain contained in an $n-1$ dimensional hypersurface with $\bar{D}$ is compact and $0 \notin \bar{D}$. We also suppose that the basic assumptions $(A0\text{-}1)$ and $(A0\text{-}2)$ hold.

From \cite{GS15}, we know that in the case $\kappa < -1$, the refracting oval is given by
\begin{equation*}
  \Gamma(P,b) := \left\{h(x,P,b)x; ~x \in S^{n-1},~x \cdot P \geq \dfrac{b + \sqrt{(\kappa^{2}-1)(\kappa^{2}\vert P \vert^{2}-b^{2})}}{\kappa^{2}}\right\},
\end{equation*}
where $h(x,P,b) = \dfrac{(\kappa^{2} x \cdot P - b) - \sqrt{(\kappa^{2} x \cdot P - b)^{2} - (\kappa^{2}-1)(\kappa^{2}\vert P \vert^{2}-b^{2})}}{\kappa^{2} - 1}.$ 

Now define
\begin{equation}\label{3.1}
  I(P,b) :=  \dfrac{b + \sqrt{(\kappa^{2}-1)(\kappa^{2}\vert P \vert^{2}-b^{2})}}{\kappa^{2}\vert P \vert},
\end{equation}
then the refracting oval can be written as
\begin{equation*}
  \Gamma(P,b) = \{h(x,P,b)x; ~x \cdot P \geq I(P,b)\vert P \vert\}.
\end{equation*}

Now we can define the notion of the refractor for the case $\kappa < -1$.
\begin{definition}
  Let $\mathcal{R} = \{\rho(x)x; x \in \bar{\Omega}\} \subset \mathcal{C}_{r_{0}}$ be a surface. Then $\mathcal{R}$ is called a near field refractor for the case $\kappa < -1$, if for any $\rho(x_{0})x_{0} \in \mathcal{R}$ there exist $P \in \bar{D}$ and a number $b$ satisfies $\kappa \vert P \vert < b < \vert P \vert$, such that the refracting oval $\Gamma(P,b)$ supports $\mathcal{R}$ at $\rho(x_{0})x_{0}$. Namely, for any $ x \in \bar{\Omega}$, we have $\rho(x) \geq h(x,P,b)$, $\rho(x_{0}) = h(x_{0},P,b)$ and $\bar{\Omega} \subset \left\{x \in S^{n-1};~ x \cdot \dfrac{P}{\vert P \vert} \geq I(P,b)\right\}$.
  \label{def3.1}
\end{definition}

The following Figure \ref{fig3} shows the refracting oval which refracts all ray emitted from the source $O$ to a specific point $P$ for the case $\kappa < -1$.

\begin{figure}[h]
  \centering
  \includegraphics[width=8.3cm]{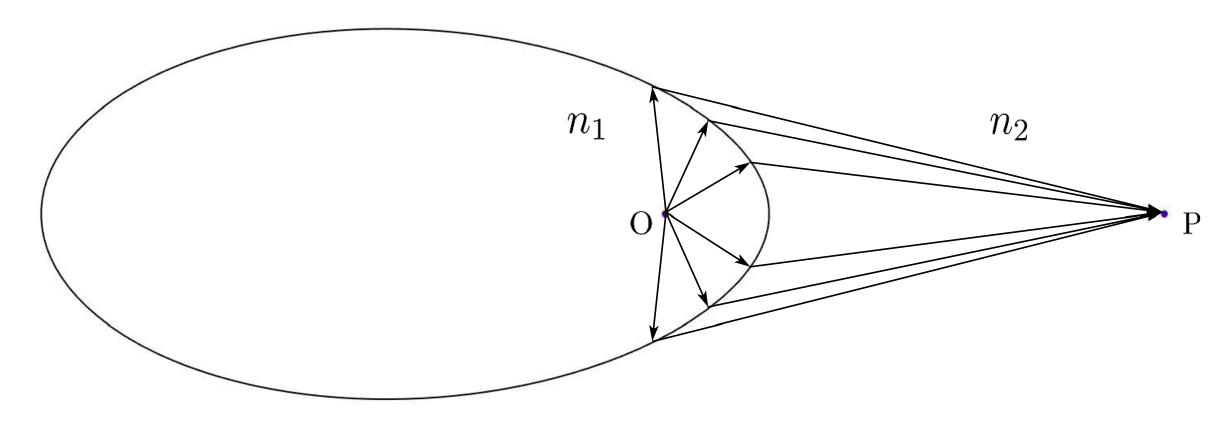}
\caption{Refracting oval when $\kappa < -1$, where $O$ is the focus of oval.}\label{fig3}
\end{figure}

In order to discuss the properties of the near field refractor for the case $\kappa < -1$, we first need the estimate of $h(x,P,b)$ and the Lipschitz continuity of $\rho$, which all stem from \cite{St17}, see the following Lemmas \ref{lem3.1} and \ref{lem3.2}.

\begin{lemma}[Lemma 3.1 in \cite{St17}]
  Suppose that $\kappa < -1$ and $\kappa \vert P \vert < b < \vert P \vert$, then we have
  \newline (a) $\min\limits_{x \in S^{n - 1}}h(x,P,b) = \dfrac{\kappa \vert P \vert - b}{\kappa - 1}$;
  \newline (b) $\max\limits_{x \in S^{n - 1}}h(x,P,b) = \dfrac{\sqrt{\kappa^{2} \vert P \vert^{2} - b^{2}}}{\sqrt{\kappa^{2} - 1}} \leq \sqrt{2 \vert P \vert}\sqrt{\dfrac{\kappa \vert P \vert - b}{1 + \kappa}}$;
  \newline (c) $\dfrac{b - \vert P \vert}{\kappa - 1} \leq \vert P -h(x,P,b)x \vert \leq \dfrac{b - \vert P \vert}{\kappa}$ for any $x \in \Gamma(P,b)$;
  \newline (d) $\dfrac{\sqrt{b - \kappa \vert P \vert}}{- \kappa \vert P \vert \sqrt{1 - \kappa}} \leq I(P,b) - \dfrac{1}{\kappa} \leq \dfrac{(1 + \sqrt{2})\sqrt{b - \kappa \vert P \vert}}{-\kappa \sqrt{\vert P \vert}}\sqrt{1 - \kappa}$.
  \label{lem3.1}
\end{lemma}

\begin{lemma}[Lemma 3.4 in \cite{St17}]
  Suppose that $\mathcal{R} = \{\rho(x)x; x \in \bar{\Omega}\}$ is a near field refractor for the case $\kappa < -1$, and for any $x \in \bar{\Omega}$, there holds $\rho(x) \leq r_{0}$, then $\rho$ is Lipschitz continuous, with the Lipschitz constant only depending on $\tau, \kappa, \vert P \vert$ and $r_{0}$.
  \label{lem3.2}
\end{lemma}

Next, we define refractor mapping and trace mapping.

\begin{definition}
  Suppose the near field refractor $\mathcal{R} = \{\rho(x)x; x \in \bar{\Omega}\}$ is given, the refractor mapping of $\mathcal{R}$ is a multi-value map define by
  \begin{equation}\label{3.2}
    \mathcal{N}_{\mathcal{R}}(x) := \{P \in \bar{D};~ \Gamma(P,b)~ supports~ \mathcal{R} ~at ~ \rho(x)x\}.
  \end{equation}
  Given $P \in \bar{D}$, the trace mapping of $\mathcal{R}$ is defined by
  \begin{equation}\label{3.3}
    \mathcal{T}_{\mathcal{R}}(P) := \{x \in \bar{\Omega};~ P \in \mathcal{N}_{\mathcal{R}}(x)\}.
  \end{equation}
  \label{def3.2}
\end{definition}

\begin{remark}
  If $\mathcal{R}$ is a near field refractor, then $\mathcal{N}_{\mathcal{R}}(\bar{\Omega}) = \bar{D}$.
  \label{rem3.1}
\end{remark}

\subsection{Case $-1 < \kappa <0$}\label{Section 3.2}

Similar to Section \ref{Section 3.1}, we start with some conditions for $\Omega$ and $D$.

Let $\Omega \subset S^{n-1}$ be a domain, satisfies $\vert \partial \Omega \vert = 0$. Let $D \subset \mathbb{R}^{n}$ be the target domain contained in an $n-1$ dimensional hypersurface with $\bar{D}$ is compact and $0 \notin \bar{D}$. We also suppose that the basic assumptions $(B0\text{-}1)$ and $(B0\text{-}2)$ hold.

From \cite{GS15}, we know that in the case $\kappa < -1$, the refracting oval is given by
\begin{equation*}
  \mathcal{O}(P,b) := \{h(x,P,b)x;~x \in S^{n-1},~ x \cdot P \geq b\},
\end{equation*}
where $h(x,P,b) = \dfrac{(b - \kappa^{2} x \cdot P) + \sqrt{(b - \kappa^{2} x \cdot P)^{2} - (1 - \kappa^{2})(b^{2} - \kappa^{2}\vert P \vert^{2})}}{1 - \kappa^{2}}.$

Now we can define the notion of the refractor for the case $-1< \kappa < 0$.

\begin{definition}
  Let $\mathcal{R} = \{\rho(x)x; x \in \bar{\Omega}\} \subset \mathcal{C}_{r_{0}}$ be a surface. Then $\mathcal{R}$ is called a near field refractor for the case $-1< \kappa < 0$, if for any $\rho(x_{0})x_{0} \in \mathcal{R}$ there exist $P \in \bar{D}$ and a number $b$ satisfies $\kappa \vert P \vert < b < \vert P \vert$, such that the refracting oval $\mathcal{O}(P,b)$ supports $\mathcal{R}$ at $\rho(x_{0})x_{0}$. Namely, for any $ x \in \bar{\Omega}$, we have $\rho(x) \leq h(x,P,b)$ and $\rho(x_{0}) = h(x_{0},P,b)$. \label{def3.3}
\end{definition}

The following Figure \ref{fig4} shows the refracting oval which refracts all ray emitted from the source $O$ to a specific point $P$ for the case $ -1 < \kappa < 0$.

\begin{figure}[h]
  \centering
  \includegraphics[width=8.3cm]{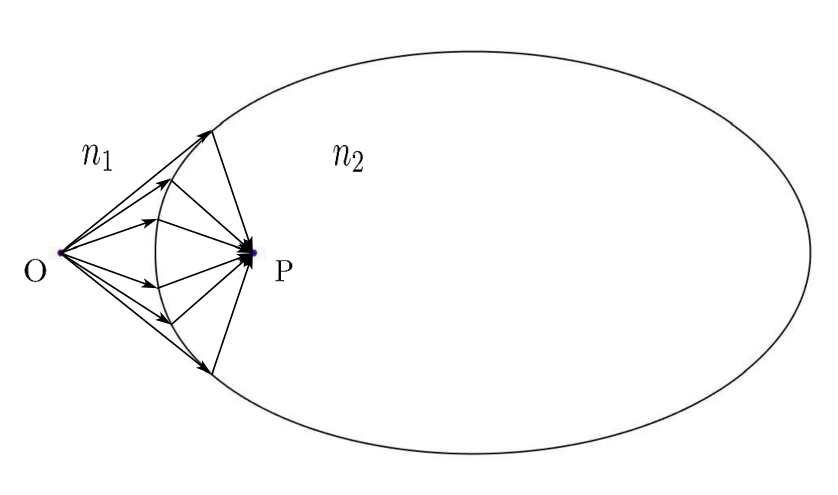}
\caption{Refracting oval when $-1 < \kappa < 0$, where $P$ is the focus of oval.}\label{fig4}
\end{figure}

Similar to Section \ref{Section 3.1}, in order to discuss the properties of the near field refractor for the case $-1< \kappa < 0$, we also need the estimate of $h(x,P,b)$ and the Lipschitz continuity of $\rho$, which all stem from \cite{St17}, see the following Lemmas \ref{lem3.3} and \ref{lem3.4}.

\begin{lemma}[Lemma 2.1 in \cite{St17}]
  Suppose that $-1< \kappa < 0$ and $\kappa \vert P \vert < b < \vert P \vert$, then we have
  \newline (a) $\min\limits_{x \in S^{n - 1}}h(x,P,b) = \dfrac{b - \kappa \vert P \vert}{1 - \kappa}$;
  \newline (b) $\max\limits_{x \in S^{n - 1}}h(x,P,b) = \dfrac{b - \kappa \vert P \vert}{1 + \kappa}$;
  \newline (c) $\min\limits_{x \in S^{n - 1}}\vert P - h(x,P,b) \vert = \dfrac{\vert P \vert - b}{1 - \kappa}$;
  \newline (d) $\max\limits_{x \in S^{n - 1}}\vert P - h(x,P,b) \vert = \dfrac{\sqrt{\vert P \vert^{2} - b^{2}}}{\sqrt{1 - \kappa^{2}}}$.
  \label{lem3.3}
\end{lemma}

\begin{lemma}[Lemma 2.3 in \cite{St17}]
  Suppose that $\mathcal{R} = \{\rho(x)x; x \in \bar{\Omega}\}$ is a near field refractor for the case $-1< \kappa < 0$, and for any $x \in \bar{\Omega}$, there holds $\rho(x) \leq r_{0}$, then $\rho$ is Lipschitz continuous, with the Lipschitz constant only depending on $\tau, \kappa$ and $\vert P \vert$.
  \label{lem3.4}
\end{lemma}

Next, we define refractor mapping and trace mapping.

\begin{definition}
  Suppose the near field refractor $\mathcal{R} = \{\rho(x)x; x \in \bar{\Omega}\}$ is given, the refractor mapping of $\mathcal{R}$ is a multi-value map define by
  \begin{equation}\label{3.4}
    \mathcal{N}_{\mathcal{R}}(x) := \{P \in \bar{D};~ \mathcal{O}(P,b)~ supports~ \mathcal{R} ~at ~ \rho(x)x\}.
  \end{equation}
  Given $P \in \bar{D}$, the trace mapping of $\mathcal{R}$ is defined by
  \begin{equation}\label{3.5}
    \mathcal{T}_{\mathcal{R}}(P) := \{x \in \bar{\Omega};~ P \in \mathcal{N}_{\mathcal{R}}(x)\}.
  \end{equation}
  \label{def3.4}
\end{definition}

\begin{remark}
  If $\mathcal{R}$ is a near field refractor, then $\mathcal{N}_{\mathcal{R}}(\bar{\Omega}) = \bar{D}$.
  \label{rem3.2}
\end{remark}

\subsection{Properties of the refractor}\label{Section 3.3}

In this subsection, we give and prove some properties of the refractor.

\begin{lemma}
  Suppose that $P \in \bar{D}$, then $\mathcal{T}_{\mathcal{R}}(P)$ is a closed set in $\bar{\Omega}$.
  \label{lem3.5}
\end{lemma}

\begin{proof}
  Let $x_{n} \in \mathcal{T}_{\mathcal{R}}(P)$ with $x_{n} \rightarrow x_{0}$, we need to prove that $x_{0} \in \mathcal{T}_{\mathcal{R}}(P)$.
  
  For the case $\kappa < -1$, because $x_{n} \in \mathcal{T}_{\mathcal{R}}(P)$, then there exists a number $b$ satisfies $\kappa \vert P \vert < b < \vert P \vert$, such that the refracting oval $\Gamma(P,b)$ supports $\mathcal{R}$ at $\rho(x_{n})x_{n}$, then we have $\rho(x_{n}) = h(x_{n},P,b)$. For $x_{0} \in \bar{\Omega}$, Lemma \ref{lem3.2} implies that $\rho(x_{n})x_{n} \rightarrow \rho(x_{0})x_{0}$, from which we deduce $\rho(x_{0}) = h(x_{0},P,b)$. Furthermore, Definition \ref{def3.1} gives $\rho(x) \geq h(x_{n},P,b)$, and thus it follows that $\rho(x) \geq h(x_{0},P,b)$. So $\Gamma(P,b)$ supports $\mathcal{R}$ at $\rho(x_{0})x_{0}$, hence $x_{0} \in \mathcal{T}_{\mathcal{R}}(P)$.
  
  For the case $-1 < \kappa < 0$, also because  $x_{n} \in \mathcal{T}_{\mathcal{R}}(P)$, then there exists a number $b$ satisfies $\kappa \vert P \vert < b < \vert P \vert$, such that the refracting oval $\mathcal{O}(P,b)$ supports $\mathcal{R}$ at $\rho(x_{n})x_{n}$, then we have $\rho(x_{n}) = h(x_{n},P,b)$. For $x_{0} \in \bar{\Omega}$, Lemma \ref{lem3.4} implies that $\rho(x_{n})x_{n} \rightarrow \rho(x_{0})x_{0}$, from which we deduce $\rho(x_{0}) = h(x_{0},P,b)$. Furthermore, Definition \ref{def3.3} gives $\rho(x) \leq h(x_{n},P,b)$, and thus it follows that $\rho(x) \leq h(x_{0},P,b)$. So $\mathcal{O}(P,b)$ supports $\mathcal{R}$ at $\rho(x_{0})x_{0}$, hence $x_{0} \in \mathcal{T}_{\mathcal{R}}(P)$.
\end{proof}

\begin{lemma}
  For any $F \subset \bar{D}$, we have
  \newline (a) $[\mathcal{T}_{\mathcal{R}}(F)]^{c} \subset \mathcal{T}_{\mathcal{R}}(F^{c})$;
  \newline (b) The set $\mathcal{M} = \{F \subset \bar{D};~\mathcal{T}_{\mathcal{R}}(F)~is~Lebesgue~measurable\}$ is a $\sigma$-algebra containing all Borel sets in $\bar{D}$.
  \label{lem3.6}
\end{lemma}

\begin{proof}
  $(a)$ If $x \in [\mathcal{T}_{\mathcal{R}}(F)]^{c}$, then $\mathcal{N}_{\mathcal{R}}(x) \cap F = \emptyset$, then $\mathcal{N}_{\mathcal{R}}(x) \cap F^{c} \neq \emptyset$. Then by the definition of trace mapping (\ref{3.3}) and (\ref{3.5}), we have $x \in \mathcal{T}_{\mathcal{R}}(F^{c})$. 
  
  $(b)$ First, we need to prove that the set $\mathcal{M}$ is a $\sigma$-algebra.
  
  Obviously, we have $\mathcal{T}_{\mathcal{R}}(\emptyset) = \emptyset$, $\mathcal{T}_{\mathcal{R}}(\bar{D}) = \bar{\Omega}$.
  For $F_{i} \in \mathcal{M}$, we have
  \begin{equation*}
      \begin{aligned}
          \mathcal{T}_{\mathcal{R}}(\bigcup\limits_{i=1}^{\infty}F_{i}) & = \bigcup\limits_{P \in \bigcup\limits_{i=1}^{\infty}F_{i}}\{x \in \bar{\Omega};~P\in \mathcal{N}_{\mathcal{R}}(x)\} \\
          & = \bigcup\limits_{i=1}^{\infty}\bigcup\limits_{P \in F_{i}}\{x \in \bar{\Omega};~P\in \mathcal{N}_{\mathcal{R}}(x)\} =\bigcup\limits_{i=1}^{\infty} \mathcal{T}_{\mathcal{R}}(F_{i}),
      \end{aligned}
  \end{equation*}
  so $\mathcal{M}$ is closed under countable union. Also for $F \in \mathcal{M}$, using $(a)$, we can get
  \begin{equation*}
    \begin{aligned}
    \mathcal{T}_{\mathcal{R}}(F^{c}) & = \{x \in \bar{\Omega};~\mathcal{N}_{\mathcal{R}}(x) \cap F^{c} \neq \emptyset\} \\
    & = \{x \in \bar{\Omega};~\mathcal{N}_{\mathcal{R}}(x) \cap F = \emptyset\} \cup \{x \in \bar{\Omega};~\mathcal{N}_{\mathcal{R}}(x) \cap F^{c} \neq \emptyset,~\mathcal{N}_{\mathcal{R}}(x) \cap F \neq \emptyset\} \\
    & = [\mathcal{T}_{\mathcal{R}}(F)]^{c} \cup [\mathcal{T}_{\mathcal{R}}(F^{c}) \cap \mathcal{T}_{\mathcal{R}}(F)].
    \end{aligned}
  \end{equation*}
  Since $\vert \mathcal{T}_{\mathcal{R}}(F^{c}) \cap \mathcal{T}_{\mathcal{R}}(F) \vert = 0$ and $\mathcal{T}_{\mathcal{R}}(F)$ is measurable, then $\mathcal{T}_{\mathcal{R}}(F^{c})$ is measurable, hence $\mathcal{M}$ is closed under taking complements. So we have the set $\mathcal{M}$ is a $\sigma$-algebra.
  
  The proof of $\mathcal{M}$ contains all Borel sets in $\bar{D}$ is similar as the proof of Lemma 3.3 in \cite{St16}.
\end{proof}

\begin{lemma}
  Suppose that $\Gamma(P_{k},b_{k})$ is a sequence of refracting oval with $P_{k} \rightarrow P_{0}$ and $b_{k} \rightarrow b_{0}$ as $k \rightarrow \infty$. If $z_{k} \in \Gamma(P_{k},b_{k})$ with $z_{k} \rightarrow z_{0}$ as $k \rightarrow \infty$, then $z_{0} \in \Gamma(P_{0},b_{0})$, and the normal $\nu_{k}(z_{k})$ to the refracting oval $\Gamma(P_{k},b_{k})$ at $z_{k}$ satisfies $\nu_{k}(z_{k}) \rightarrow \nu(z_{0})$, where $\nu(z_{0})$ is the normal to the refracting oval $\Gamma(P_{0},b_{0})$ at the point $z_{0}$.
  \label{lem3.7}
\end{lemma}

\begin{proof}
  Let $z = h(x,P_{k},b_{k})x$, $x \in S^{n-1}$, then the Cartesian coordinate of the equation of the refracting oval is $\vert z \vert + \kappa \vert z - P_{k} \vert = b_{k}$, so the normal vector at $z$ is $\nu_{k}(z) = \dfrac{z}{\vert z \vert} - \kappa \dfrac{P_{k} - z}{\vert P_{k} - z \vert}$. So we have
  $$\nu_{k}(z_{k}) = \dfrac{z_{k}}{\vert z_{k} \vert} - \kappa \dfrac{P_{k} - z_{k}}{\vert P_{k} - z_{k} \vert} \rightarrow \dfrac{z_{0}}{\vert z_{0} \vert} - \kappa \dfrac{P_{0} - z_{0}}{\vert P_{0} - z_{0} \vert} = \nu(z_{0}).$$
\end{proof}

\begin{remark}
   Lemma \ref{lem3.7} still holds for the case $-1 < \kappa < 0$, with  $\Gamma(P_{k},b_{k})$ replaced by $\mathcal{O}(P_{k},b_{k})$.
  \label{rem3.3}
\end{remark}

\begin{lemma}
  Assume that $\mathcal{R}_{k} = \{\rho_{k}(x)x; ~x\in \bar{\Omega}\}$, $k \geq 1$, is a sequence of near field refractors, then fix $x_{0} \in \bar{\Omega}$, there exist $P_{k} \in \bar{D}$ and $b_{k}$, such that $\Gamma(P_{k},b_{k})$ (for the case $\kappa < -1$) or $\mathcal{O}(P_{k},b_{k})$ (for the case $-1 < \kappa < 0$) supports $\mathcal{R}_{k}$ at $\rho(x_{0})x_{0}$.  Suppose that there exist $C_{1}, C_{2} > 0$, such that $C_{1} \leq \rho_{k}(x) \leq r_{0}$ in $\bar{\Omega}$ for the case $\kappa < -1$ and $C_{2} \leq \rho_{k}(x) \leq r_{0}$ in $\bar{\Omega}$ for the case $-1 < \kappa < 0$, then $b_{k}$ is bounded. Moreover, we have
  \begin{equation}\label{3.6}
    C_{1}(1 - \kappa) + \kappa \vert P_{k} \vert \leq b_{k} \leq \sqrt{\kappa^{2}(\vert P_{k} \vert^{2} - C_{1}^{2}) + C_{1}^{2}}
  \end{equation}
  for the case $\kappa < -1$, and
  \begin{equation}\label{3.7}
    C_{2} (1 + \kappa) + \kappa \vert P_{k} \vert \leq b_{k} \leq \kappa \vert P_{k} \vert + (1 - \kappa) r_{0}
  \end{equation}
  for the case $-1 < \kappa < 0$.
  \label{lem3.8}
\end{lemma}

\begin{proof}
  For the case $\kappa < -1$, because $\Gamma(P_{k},b_{k})$ supports $\mathcal{R}_{k}$ at $\rho(x_{0})x_{0}$, then
  $$\rho_{k}(x) \geq h(x,P_{k},b_{k}) \quad \text{and} \quad \rho_{k}(x_{0}) = h(x_{0},P_{k},b_{k})~\text{for all} ~x\in \bar{\Omega}.$$
  Since
  $$\kappa \vert P - h(x,P,b)x \vert = b - h(x,P,b),$$
  we have
  $$\rho_{k}(x) + \kappa \vert P_{k} - h(x,P_{k},b_{k})x \vert \geq b_{k} \quad \text{and} \quad \rho_{k}(x_{0}) + \kappa \vert P_{k} - h(x_{0},P_{k},b_{k})x_{0} \vert = b_{k}.$$
  
  On the one hand, from Lemma \ref{lem3.1} $(c)$, we have
  $$b_{k} = \rho_{k}(x_{0}) + \kappa \vert P_{k} - h(x_{0},P_{k},b_{k}) \vert \geq C_{1} + \kappa \dfrac{b_{k} - \vert P_{k} \vert}{\kappa - 1},$$
  from which we deduce
  \begin{equation}\label{3.8}
    b_{k} \geq C_{1}(1 - \kappa) + \kappa \vert P_{k} \vert.
  \end{equation}
  On the other hand, from Lemma \ref{lem3.1} $(b)$, we have 
  $$h^{2}(x,P_{k},b_{k}) \leq \dfrac{\kappa^{2} \vert P_{k} \vert^{2} - b_{k}^{2}}{\kappa^{2} - 1}.$$ 
  Hence
  \begin{equation*}
    \begin{aligned}
      b_{k}^{2} &\leq \kappa^{2} \vert P_{k} \vert^{2} - (\kappa^{2} - 1) h^{2}(x,P_{k},b_{k}) \leq \kappa^{2} \vert P_{k} \vert^{2} - (\kappa^{2} - 1)\rho_{k}^{2}(x_{k}) \\
      &\leq \kappa^{2} \vert P_{k} \vert^{2} - (\kappa^{2} - 1) C_{1}^{2}.
    \end{aligned}
  \end{equation*}
  Then we obtain
  \begin{equation}\label{3.9}
    b_{k} \leq \sqrt{\kappa^{2}(\vert P_{k} \vert^{2} - C_{1}^{2}) + C_{1}^{2}}.
  \end{equation}
  Combining (\ref{3.8}) with (\ref{3.9}), we have
  \begin{equation*}
    C_{1}(1 - \kappa) + \kappa \vert P_{k} \vert \leq b_{k} \leq \sqrt{\kappa^{2}(\vert P_{k} \vert^{2} - C_{1}^{2}) + C_{1}^{2}},
  \end{equation*}
  so $b_{k}$ is bounded for the case $\kappa < -1$.
  
  For the case $-1 < \kappa < 0$, because $\mathcal{O}(P_{k},b_{k})$ supports $\mathcal{R}_{k}$ at $\rho(x_{0})x_{0}$, then
  $$\rho_{k}(x) \leq h(x,P_{k},b_{k}) \quad \text{and} \quad \rho_{k}(x_{0}) = h(x_{0},P_{k},b_{k})~\text{for all} ~x\in \bar{\Omega},$$
  so we have 
  $$\rho_{k}(x) + \kappa \vert P_{k} - h(x,P_{k},b_{k})x \vert \leq b_{k} \quad \text{and} \quad \rho_{k}(x_{0}) + \kappa \vert P_{k} - h(x_{0},P_{k},b_{k})x_{0} \vert = b_{k}.$$
  Then from Lemma \ref{lem3.3}, we have 
  \begin{equation*}
    C_{2} (1 + \kappa) + \kappa \vert P_{k} \vert \leq b_{k} \leq \kappa \vert P_{k} \vert + (1 - \kappa) r_{0},
  \end{equation*}
  so $b_{k}$ is bounded for the case $-1 < \kappa < 0$.
\end{proof}

\begin{lemma}
   Suppose that the conditions of the Lemma \ref{lem3.8} holds, and $\rho_{k} \rightarrow \rho$ uniformly on $\bar{\Omega}$, then we have
  \newline (a) $\mathcal{R} := \{\rho(x)x; ~x\in \bar{\Omega}\}$ is also a near field refractor;
  \newline (b) For any compact set $K\subset \bar{D}$,
  $$\varlimsup_{k\rightarrow \infty} \mathcal{T}_{\mathcal{R}_{k}}(K) \subset \mathcal{T}_{\mathcal{R}}(K).$$
  \newline (c) For any open set $G \subset \bar{D}$,
  $$\mathcal{T}_{\mathcal{R}}(G) \subset \varliminf_{k\rightarrow \infty}\mathcal{T}_{\mathcal{R}_{k}}(G) \cup E,$$
  where $E$ is the singular set of $\mathcal{R}$.
  \label{lem3.9}
\end{lemma}

\begin{proof}
  We only consider the case $\kappa < -1$. The case $-1 < \kappa < 0$ is similar.
  
  $(a)$ From Lemma \ref{lem3.8}, by taking a subsequence if necessary, we may assume that there exist $P_{0} \in \bar{D}$ and $b_{0}$, such that $P_{k} \rightarrow P_{0}$ and $b_{k} \rightarrow b_{0}$ as $k \rightarrow \infty$. Hence we have
  $$\rho(x_{0}) = \lim\limits_{k \rightarrow \infty}\rho_{k}(x_{0}) = \lim\limits_{k \rightarrow \infty}h(x_{0},P_{k},b_{k}) = h(x_{0},P_{0},b_{0})$$
  and
  $$\rho(x) = \lim\limits_{k \rightarrow \infty}\rho_{k}(x) \geq \lim\limits_{k \rightarrow \infty}h(x,P_{k},b_{k}) = h(x,P_{0},b_{0})~\text{for all} ~x\in \bar{\Omega}.$$
  So $\Gamma(P_{0},b_{0})$ supports $\mathcal{R}$ at $\rho(x_{0})x_{0}$, then $\mathcal{R}$ is a near field refractor.
  
  $(b)$ Let $x_{0} \in \varlimsup\limits_{k\rightarrow \infty} \mathcal{T}_{\mathcal{R}_{k}}(K)$. Without loss of generality, we assume that $x_{0} \in \mathcal{T}_{\mathcal{R}_{k_{j}}}(K)$ for $j = 1,2,\ldots$, then there exists $P_{k_{j}} \in K$ such that 
  $$\rho_{k_{j}}(x) \geq h(x,P_{k_{j}},b_{k_{j}}) \quad \text{and} \quad \rho_{k_{j}}(x_{0}) = h(x_{0},P_{k_{j}},b_{k_{j}})$$
  and
  $$x \cdot P_{k_{j}} \geq I(P_{k_{j}},b_{k_{j}}).$$
  From (\ref{3.6}), we know that there exist $P_{0} \in K$ and $b_{0} > 0$, such that $P_{k_{j}} \rightarrow P_{0}$ and $b_{k_{j}} \rightarrow b_{0}$ as $j \rightarrow \infty$. From the proof of $(a)$, $\Gamma(P_{0},b_{0})$ supports $\mathcal{R}$ at $\rho(x_{0})x_{0}$, hence $x_{0} \in \mathcal{T}_{\mathcal{R}}(P_{0})$, so $x_{0} \in \mathcal{T}_{\mathcal{R}}(K)$.
  
  $(c)$ Since $\mathcal{T}_{\mathcal{R}}(F^{c}) = [\mathcal{T}_{\mathcal{R}}(F)]^{c} \cup [\mathcal{T}_{\mathcal{R}}(F^{c}) \cap \mathcal{T}_{\mathcal{R}}(F)]$, from $(b)$, we have
  \begin{equation}\label{3.10}
  \begin{aligned}
    \varlimsup_{k \rightarrow \infty}[\mathcal{T}_{\mathcal{R}_{k}}(G)]^{c} & \subset \varlimsup_{k \rightarrow \infty}([\mathcal{T}_{\mathcal{R}_{k}}(G)]^{c} \cup [\mathcal{T}_{\mathcal{R}_{k}}(G) \cap \mathcal{T}_{\mathcal{R}_{k}}(G^{c})])\\
    & \subset \varlimsup_{k \rightarrow \infty}\mathcal{T}_{\mathcal{R}_{k}}(G^{c}) \subset \mathcal{T}_{\mathcal{R}}(G^{c}) \\
    & =[\mathcal{T}_{\mathcal{R}}(G)]^{c} \cup [\mathcal{T}_{\mathcal{R}}(G^{c}) \cap \mathcal{T}_{\mathcal{R}}(G)].
    \end{aligned}
  \end{equation}
  Since $\mathcal{T}_{\mathcal{R}}(G^{c}) \cap \mathcal{T}_{\mathcal{R}}(G) \subset E$, taking complements in (\ref{3.10}), we get
  $$\mathcal{T}_{\mathcal{R}}(G) \subset \varliminf_{k\rightarrow \infty}\mathcal{T}_{\mathcal{R}_{k}}(G) \cup E.$$
\end{proof}

\section{Fresnel formula}\label{Section 4}

\sloppy{}

\subsection{Expression of the Fresnel coefficients}\label{Section 4.1}

From the previous analysis,  we know that when the incident light ray strikes the surface $\mathcal{R}$, it will split into refracted light ray and reflected light ray, so the energy of the incident light ray will be distributed to the refracted light ray and reflected light ray. This subsection briefly gives the energy distribution of reflected and refracted light ray  according to the electromagnetic field theory of light propagation. Although the derivation of the Fresnel coefficients can be found in \cite{GS15} and \cite{SJ25}, we still present it here for the sake of completeness.

Define $\mathbf{E} = \mathbf{E}(\mathbf{r},t)$ as electric field vector and $\mathbf{B} = \mathbf{B}(\mathbf{r},t)$ as magnetic field vector, where $\mathbf{r} = \mathbf{r}(x,y,z)$ represents a point in 3-d space and $t$ is the time, then we have the following system of Maxwell's equations absent from charges:
\begin{equation}\label{4.1}
  \left\{
   \begin{split}
     & \nabla\times \mathbf{E} = -\frac{\mu}{c} \frac{\partial \mathbf{B}}{\partial t},  \\
     & \nabla\times \mathbf{B} = -\frac{\epsilon}{c} \frac{\partial \mathbf{E}}{\partial t},  \\
     & \nabla \cdot (\epsilon \mathbf{E}) = 0,  \\
     & \nabla \cdot (\mu \mathbf{B}) = 0,
   \end{split}
    \right.
\end{equation}
where $c$ is the speed of light in vacuum, $\mu = \mu(x,y,z)$ is the magnetic permeability of the medium and $\epsilon = \epsilon(x,y,z)$ is the electric permittivity of the medium.

Assume that the waves are plane waves, that is, the waves have the same value at all points of any plane perpendicular to the direction of propagation, then from (\ref{4.1}), we have:
\begin{equation}\label{4.2}
  \left\{
  \begin{split}
      & \mathbf{E} = -\frac{c}{\epsilon \omega}(\mathbf{k} \times \mathbf{B}), \\
      & \mathbf{B} = \frac{c}{\mu \omega}(\mathbf{k} \times \mathbf{E}),
  \end{split}
  \right.
\end{equation}
where $c$ is the speed of light in free space, $\mathbf{k} = \displaystyle \frac{\omega}{v} \mathbf{s}$ represents the wave vector, $\omega$ represents the angular frequency of the electromagnetic wave, $v$ represents the speed of light ray in the medium and $\mathbf{s}$ is a unit vector.

The flow of the energy in an electromagnetic wave with electric field $\mathbf{E} = \mathbf{E}(\mathbf{r},t)$ and magnetic field $\mathbf{B} = \mathbf{B}(\mathbf{r},t)$ is given by Poynting vector
\begin{equation}\label{4.3}
  \mathbf{S} = \frac{c}{4\pi}\mathbf{E} \times \mathbf{B}.
\end{equation}
Then from (\ref{4.2}), we have
\begin{equation}\label{4.4}
  \mathbf{S} = \frac{c}{4\pi}\mathbf{E} \times (\frac{c}{\mu \omega} \mathbf{k} \times \mathbf{E}) = \frac{c}{4\pi} \sqrt{\frac{\epsilon}{\mu}}  \mathbf{E} \times \mathbf{s} \times \mathbf{E}.
\end{equation}

We denote quantities referring to the incident wave by the suffix $(i)$, to the refracted wave by $(t)$ and to the reflected wave by $(r)$. Choose a system of coordinates such that the normal $\nu$ to the interface $\Gamma$ at the point of incidence is on the $z$-axis and the $x$ and $y$ axes are on the plane perpendicular to $\nu$. So the tangent plane to $\Gamma$ at $P$ is the $xy$-plane and the incident plane is the $xz$-plane. Then each of the electric field and magnetic field vectors can be resolved into components parallel denoted by subscript $\parallel$ and perpendicular denoted by subscript $\bot$. Then we obtain:
\begin{equation*}
  \left\{
  \begin{split}
        \mathbf{E}^{(\mathbf{i})}(\mathbf{r},t) & = (-A_{\parallel}\cos\theta_{i} , A_{\bot} , A_{\parallel}\sin\theta_{i}) \cos(\omega(t - \frac{\mathbf{r} \cdot \mathbf{s^{(i)}}}{v_{1}}))  \\
       & = \mathbf{E}_{0}^{(\mathbf{i})} \cos(\omega(t - \frac{\mathbf{r} \cdot \mathbf{s^{(i)}}}{v_{1}})), \\
        \mathbf{E}^{(\mathbf{t})}(\mathbf{r},t) & = (-T_{\parallel}\cos\theta_{t} , T_{\bot} , T_{\parallel}\sin\theta_{t}) \cos(\omega(t - \frac{\mathbf{r} \cdot \mathbf{s^{(t)}}}{v_{2}})) \\
        & = \mathbf{E}_{0}^{(\mathbf{t})} \cos(\omega(t - \frac{\mathbf{r} \cdot \mathbf{s^{(t)}}}{v_{2}})), \\
        \mathbf{E}^{(\mathbf{r})}(\mathbf{r},t) & = (R_{\parallel}\cos\theta_{r} , R_{\bot} , R_{\parallel}\sin\theta_{r}) \cos(\omega(t - \frac{\mathbf{r} \cdot \mathbf{s^{(r)}}}{v_{1}})) \\
        & = \mathbf{E}_{0}^{(\mathbf{r})} \cos(\omega(t - \frac{\mathbf{r} \cdot \mathbf{s^{(r)}}}{v_{1}})),
  \end{split}
  \right.
\end{equation*}
and
\begin{equation*}
  \begin{cases}
    \mathbf{B}^{(\mathbf{i})}(\mathbf{r},t) & = \sqrt{\dfrac{\epsilon_{1}}{\mu_{1}}}(-A_{\bot}\cos\theta_{i} , -A_{\parallel} , A_{\bot}\sin\theta_{i}) \cos(\omega(t - \dfrac{\mathbf{r} \cdot \mathbf{s^{(i)}}}{v_{1}})) \\
    & =  \sqrt{\dfrac{\epsilon_{1}}{\mu_{1}}} \mathbf{B}_{0}^{(\mathbf{i})} \cos(\omega(t - \dfrac{\mathbf{r} \cdot \mathbf{s^{(i)}}}{v_{1}})), \\
    \mathbf{B}^{(\mathbf{t})}(\mathbf{r},t) & = \sqrt{\dfrac{\epsilon_{2}}{\mu_{2}}}(-A_{\bot}\cos\theta_{t} , -A_{\parallel} , A_{\bot}\sin\theta_{t}) \cos(\omega(t - \dfrac{\mathbf{r} \cdot \mathbf{s^{(t)}}}{v_{2}})) \\
    & =  \sqrt{\dfrac{\epsilon_{2}}{\mu_{2}}} \mathbf{B}_{0}^{(\mathbf{t})} \cos(\omega(t - \dfrac{\mathbf{r} \cdot \mathbf{s^{(t)}}}{v_{2}})), \\
    \mathbf{B}^{(\mathbf{r})}(\mathbf{r},t) & = \sqrt{\dfrac{\epsilon_{1}}{\mu_{1}}}(-A_{\bot}\cos\theta_{r} , -A_{\parallel} , A_{\bot}\sin\theta_{r}) \cos(\omega(t - \dfrac{\mathbf{r} \cdot \mathbf{s^{(r)}}}{v_{1}}))  \\
    & =  \sqrt{\dfrac{\epsilon_{1}}{\mu_{1}}} \mathbf{B}_{0}^{(\mathbf{r})} \cos(\omega(t - \dfrac{\mathbf{r} \cdot \mathbf{s^{(r)}}}{v_{1}})),
  \end{cases}
\end{equation*}
where $v_{1} = \displaystyle \frac{c}{\sqrt{\epsilon_{1} \mu_{1}}}$, $v_{2} = \displaystyle \frac{c}{\sqrt{\epsilon_{2} \mu_{2}}}$, $A$, $R$ and $T$ are the amplitude vectors and $\mathbf{s}^{(\mathbf{i})}$, $\mathbf{s}^{(\mathbf{t})}$ and $\mathbf{s}^{(\mathbf{r})}$ are the directions of propagation of the corresponding fields. The boundary conditions expressing the continuity of the tangential components of the electric and magnetic fields across the interface \cite{BW13}, then we have
\begin{equation}\label{4.5}
  \left\{
  \begin{split}
       & \mathbf{k} \times \mathbf{E}^{(\mathbf{i})}_{0} + \mathbf{k} \times \mathbf{E}^{(\mathbf{r})}_{0} =  \mathbf{k} \times \mathbf{E}^{(\mathbf{t})}_{0},  \\
       & \mathbf{k} \times \mathbf{B}^{(\mathbf{i})}_{0} + \mathbf{k} \times \mathbf{B}^{(\mathbf{r})}_{0} = \mathbf{k} \times \mathbf{B}^{(\mathbf{t})}_{0}.
  \end{split}
  \right.
\end{equation}
From (\ref{4.4}), we obtain
\begin{equation}\label{4.6}
  \left\{
  \begin{split}
       & A_{\bot} + R_{\bot} = T_{\bot}, \\
       & \cos\theta_{i}(A_{\parallel} - R_{\parallel}) = \cos\theta_{t}T_{\parallel}, \\
       & \frac{A_{\parallel}}{\sqrt{\dfrac{\epsilon_{1}}{\mu_{1}}}} +  \frac{R_{\parallel}}{\sqrt{\dfrac{\epsilon_{1}}{\mu_{1}}}} = \frac{T_{\parallel}}{\sqrt{\dfrac{\epsilon_{2}}{\mu_{2}}}}, \\
       & \cos\theta_{i}\left(\frac{A_{\bot}}{\sqrt{\dfrac{\epsilon_{1}}{\mu_{1}}}} - \frac{R_{\bot}}{\sqrt{\dfrac{\epsilon_{1}}{\mu_{1}}}}\right) = \cos\theta_{t} \frac{T_{\bot}}{\sqrt{\dfrac{\epsilon_{2}}{\mu_{2}}}}.
  \end{split}
  \right.
\end{equation}
Define the wave impedance of the medium as $z = \sqrt{\dfrac{\mu}{\epsilon}}$, then we obtain the following Fresnel formula from (\ref{4.6}):
\begin{equation}\label{4.7}
  \left\{
  \begin{split}
       & T_{\parallel} = \frac{2z_{1}\cos\theta_{i}}{z_{2} \cos\theta_{i} + z_{1}\cos\theta_{t}}A_{\parallel}, \\
       & T_{\bot} = \frac{2z_{1}\cos\theta_{i}}{z_{1}\cos\theta_{i} +  z_{2} \cos\theta_{t}}A_{\bot}, \\
       & R_{\parallel} =  \frac{z_{2} \cos\theta_{i} - z_{1}\cos\theta_{t}}{z_{2} \cos\theta_{i} + z_{1}\cos\theta_{t}}A_{\parallel},\\
       & R_{\bot} = \frac{z_{1}\cos\theta_{i} -  z_{2} \cos\theta_{t}}{z_{1}\cos\theta_{i} +  z_{2} \cos\theta_{t}}A_{\bot}.
  \end{split}
  \right.
\end{equation}
From Snell law (\ref{2.3}) and the fact that $x \cdot \nu = \cos\theta_{i}$, $m \cdot \nu = \cos\theta_{t}$, (\ref{4.7}) can be written as
\begin{equation}\label{4.8}
  \left\{
  \begin{split}
       & T_{\parallel} = \frac{2 z_{1} x\cdot (x-\kappa m)}{(z_{2} x + z_{1} m)\cdot (x - \kappa m)}A_{\parallel}, \\
       & T_{\bot} = \frac{2 z_{1} x\cdot (x-\kappa m)}{(z_{1} x + z_{2} m) \cdot (x - \kappa m)}A_{\bot}, \\
       & R_{\parallel} = \frac{(z_{2} x - z_{1} m)\cdot (x - \kappa m)}{(z_{2} x + z_{1} m)\cdot (x - \kappa m)}A_{\parallel}, \\
       & R_{\bot} = \frac{(z_{1} x - z_{2} m) \cdot (x - \kappa m)}{(z_{1} x + z_{2} m) \cdot (x - \kappa m)}A_{\bot}.
  \end{split}
  \right.
\end{equation}

Using Poynting vector (\ref{4.3}), the amount of energies of incident, transmitted and reflected waves leaving a unit area of the boundary per second is given by
\begin{equation*}
  \left\{
  \begin{split}
       & J^{(i)} = \vert \mathbf{S}^{\mathbf{i}} \vert \cos\theta_{i} = \frac{c}{4\pi} \sqrt{\frac{\epsilon_{1}}{\mu_{1}}} \vert \mathbf{E}_{0}^{(\mathbf{i})}\vert^{2} x \cdot \nu, \\
       & J^{(t)} = \vert \mathbf{S}^{\mathbf{t}} \vert \cos\theta_{t} = \frac{c}{4\pi} \sqrt{\frac{\epsilon_{2}}{\mu_{2}}} \vert \mathbf{E}_{0}^{(\mathbf{t})}\vert^{2} m \cdot \nu, \\
       & J^{(r)} = \vert \mathbf{S}^{\mathbf{r}} \vert \cos\theta_{r} = \frac{c}{4\pi} \sqrt{\frac{\epsilon_{1}}{\mu_{1}}} \vert \mathbf{E}_{0}^{(\mathbf{r})}\vert^{2} x \cdot \nu.
  \end{split}
  \right.
\end{equation*}
Then we can define the reflection and transmission coefficients as
\begin{equation*}
  \left\{
  \begin{split}
       & r_{\mathcal{R}}(x) = \frac{J^{(r)}}{J^{(i)}} = \left(\frac{\vert \mathbf{E}_{0}^{(\mathbf{r})} \vert}{\vert \mathbf{E}_{0}^{(\mathbf{i})} \vert}\right)^{2}, \\
       & t_{\mathcal{R}}(x) = \frac{J^{(t)}}{J^{(i)}} = \sqrt{\frac{\epsilon_{2}\mu_{1}}{\epsilon_{1}\mu_{2}}} \frac{m \cdot \nu}{x \cdot \nu}\left(\frac{\vert \mathbf{E}_{0}^{(\mathbf{t})} \vert}{\vert \mathbf{E}_{0}^{(\mathbf{i})}\vert}\right)^{2}.
  \end{split}
  \right.
\end{equation*}
From Fresnel formula (\ref{4.8}), we have
\begin{equation}\label{4.9}
  \begin{aligned}
  r_{\mathcal{R}}(x) & = \left[ \frac{z_{2} + \kappa z_{1} - (z_{1} + \kappa z_{2}) x \cdot m}{z_{2} - \kappa z_{1} + (z_{1} - \kappa z_{2}) x \cdot m} \right]^{2} \frac{A_{\parallel}^{2}}{A_{\parallel}^{2} + A_{\bot}^{2}} \\
  & \quad + \left[ \frac{z_{1} + \kappa z_{2} - (z_{2} + \kappa z_{1}) x \cdot m}{z_{1} - \kappa z_{2} + (z_{2} - \kappa z_{1}) x \cdot m} \right]^{2} \frac{A_{\bot}^{2}}{A_{\parallel}^{2} + A_{\bot}^{2}},
  \end{aligned}
\end{equation}
and by conservation of energy, we have
\begin{equation}\label{4.10}
  t_{\mathcal{R}}(x) = 1 - r_{\Gamma}(x).
\end{equation}

\begin{remark}
  Equations (\ref{4.9}) and (\ref{4.10}) are called Fresenl's Equation and $r_{\mathcal{R}}(x)$ and $t_{\mathcal{R}}(x)$ are called Fresnel coefficients.
  \label{rem4.1}
\end{remark}

\begin{remark}
    From Snell law (\ref{2.3}) and equations (\ref{4.9}) and (\ref{4.10}), $r_{\mathcal{R}}(x)$ and $t_{\mathcal{R}}(x)$ are functions only depending on $x$ and $\nu$.
    \label{rem4.2}
\end{remark}

\subsection{Properties of the Fresnel coefficients}\label{Section 4.2}

In this section, we discuss some properties of the Fresnel coefficients. We start with the boundedness of the Fresnel coefficients.

\begin{proposition}
  Suppose that $\bar{\Omega}$ and $\bar{D}$ satisfy the conditions in Remark \ref{rem2.1}, $\mathcal{R}$ is a refractor from $\bar{\Omega}$ to $\bar{D}$, then there exists a constant $C_{\varepsilon}$ associated with $\varepsilon$, such that $r_{\mathcal{R}}(x) \leq C_{\varepsilon}$ and $C_{\varepsilon} < t_{\mathcal{R}}(x) <1$.
  \label{prop4.1}
\end{proposition}

\begin{proof}
  For simplicity, let $\sigma = \displaystyle \frac{z_{2}}{z_{1}} = \displaystyle \sqrt{\frac{\mu_{2}\epsilon_{1}}{\mu_{1}\epsilon_{2}}} > 0$ and introduce a function
\begin{equation}\label{4.11}
  \psi(t) := \left[\frac{\sigma + \kappa -(1 + \kappa \sigma)t}{\sigma - \kappa + (1 - \kappa \sigma)t}\right]^{2} \alpha + \left[\frac{1 + \kappa \sigma - (\sigma + \kappa)t}{1 - \kappa \sigma + (\sigma - \kappa)t}\right]^{2} \beta,
\end{equation}
where $\alpha = \dfrac{A_{\parallel}^{2}}{A_{\parallel}^{2} + A_{\bot}^{2}}$, $\beta = \dfrac{A_{\bot}^{2}}{A_{\parallel}^{2} + A_{\bot}^{2}}$. Then $r_{\mathcal{R}}(x) = \psi(x \cdot m)$, $t_{\mathcal{R}}(x) = 1- \psi(x \cdot m)$. We denote
$$p(t) = \frac{\sigma + \kappa -(1 + \kappa \sigma)t}{\sigma - \kappa + (1 - \kappa \sigma)t} \quad \text{and} \quad q(t) = \frac{1 + \kappa \sigma - (\sigma + \kappa)t}{1 - \kappa \sigma + (\sigma - \kappa)t}.$$

For the case $\kappa < -1$, from (\ref{2.6}), we know $t\in\left[\dfrac{1}{\kappa} + \varepsilon,1\right]$.

For $p(t)$, we have $p'(t) = \displaystyle \frac{2 \sigma (\kappa^{2} - 1)}{[\sigma - \kappa + (1 - \kappa \sigma)t]^{2}}$.  For $\kappa <-1$, then $\kappa^{2} - 1 >0$, so $p(t)$ increases on $\left[\dfrac{1}{\kappa} + \varepsilon,1\right]$. Hence 
$$p^{2}(t)_{\max} = \max\left\{p^{2}\left(\dfrac{1}{\kappa} + \varepsilon\right), p^{2}(1)\right\}.$$
We have $p^{2}(1) = \left[\dfrac{\sigma - 1}{\sigma + 1}\right]^{2}$, $p^{2}\left(\dfrac{1}{\kappa} + \varepsilon\right) = \left[\dfrac{\kappa^{2} - 1 - \varepsilon \kappa(1 + \kappa \sigma)}{1 - \kappa^{2} + \varepsilon \kappa(1 - \kappa \sigma)}\right]^{2}$. Since $\varepsilon$ is small enough, then $p^{2}(t)_{\max} = p^{2}\left(\dfrac{1}{\kappa} + \varepsilon\right)$.

For $q(t)$, we have $q'(t) = \dfrac{2 \sigma (\kappa^{2} - 1)}{[1 - \kappa \sigma + (\sigma - \kappa)t]^{2}}$. For $\kappa <-1$, then $\kappa^{2} - 1 >0$, then $q(t)$ increases on $\left[\dfrac{1}{\kappa} + \varepsilon,1\right]$. Hence 
$$q^{2}(t)_{\max} = \max\left\{q^{2}\left(\dfrac{1}{\kappa} + \varepsilon\right), q^{2}(1)\right\}.$$
We have $q^{2}(1) = \left[-\dfrac{\sigma - 1}{\sigma + 1}\right]^{2}$, $q^{2}\left(\dfrac{1}{\kappa} + \varepsilon\right) = \left[\dfrac{\sigma (\kappa^{2} - 1) - \varepsilon \kappa^{2}(\sigma + 1)}{\sigma (1 - \kappa^{2}) + \varepsilon \kappa (\sigma - \kappa)}\right]^{2}$. Since $\varepsilon$ is small enough, then $q^{2}(t)_{\max} = q^{2}\left(\dfrac{1}{\kappa} + \varepsilon\right)$.

For the case $-1 < \kappa < 0$, from (\ref{2.7}), we know $t \in [\kappa + \varepsilon, 1]$.

For $p(t)$, we have $p'(t) = \displaystyle \frac{2 \sigma (\kappa^{2} - 1)}{[\sigma - \kappa + (1 - \kappa \sigma)t]^{2}}$.  For $-1< \kappa <0$, then $\kappa^{2} - 1 <0$, so $p(t)$ decreases on $[\kappa + \varepsilon,1]$. Hence 
$$p^{2}(t)_{\max} = \max\{p^{2}(\kappa + \varepsilon), p^{2}(1)\}.$$
We have $p^{2}(1) = \left[\dfrac{\sigma - 1}{\sigma + 1}\right]^{2}$, $p^{2}(\kappa + \varepsilon) = \left[\dfrac{\sigma(1 - \kappa^{2}) - \varepsilon (1 + \kappa \sigma)}{\sigma(1 - \kappa^{2}) + \varepsilon (1 - \kappa \sigma)}\right]^{2}$. Since $\varepsilon$ is small enough, then $p^{2}(t)_{\max} = p^{2}(\kappa + \varepsilon)$.
  
For $q(t)$, we have $q'(t) = \dfrac{2 \sigma (\kappa^{2} - 1)}{[1 - \kappa \sigma + (\sigma - \kappa)t]^{2}}$. For $-1< \kappa <0$, then $\kappa^{2} - 1 <0$, then $q(t)$ decreases on $[\kappa + \varepsilon,1]$. Hence 
$$q^{2}(t)_{\max} = \max\{q^{2}(\kappa + \varepsilon), q^{2}(1)\}.$$
We have $q^{2}(1) = \left[-\dfrac{\sigma - 1}{\sigma + 1}\right]^{2}$, $q^{2}(\kappa + \varepsilon) = \left[\dfrac{1 - \kappa^{2} - \varepsilon(\sigma + \kappa)}{1 - \kappa^{2} + \varepsilon(\sigma - \kappa)}\right]^{2}$. Since $\varepsilon$ is small enough, then $q^{2}(t)_{\max} = q^{2}(\kappa + \varepsilon)$.
  
From above analysis, there exists a constant $C_{\varepsilon}$ associated with $\varepsilon$, such that $r_{\mathcal{R}}(x) \leq C_{\varepsilon}$, and for $t_{\mathcal{R}}(x) = 1 - r_{\mathcal{R}}(x)$, then $C_{\varepsilon} < t_{\mathcal{R}}(x) <1$.
\end{proof}
 
Based on the boundedness of Fresnel coefficients, we discuss some other properties of them. The proofs for the following properties for the case $\kappa < -1$ and the case $-1 < \kappa < 0$ are also analogous, so we only prove the following properties for the case $\kappa < -1$.

\begin{proposition}
  Suppose that $\mathcal{R} = \{\rho(x)x;~x\in \bar{\Omega}\}$ is a near field refractor from $\bar{\Omega}$ to $\bar{D}$ and $E$ is the singular set of $\mathcal{R}$, then $t_{\mathcal{R}}(x)$ is continuous on $\bar{\Omega}\backslash E$.
  \label{prop4.2}
\end{proposition}

\begin{proof}
  To prove $t_{\mathcal{R}}(x)$ is continuous on $\bar{\Omega}\backslash E$, we only need to prove $r_{\mathcal{R}}(x)$ is continuous on $\bar{\Omega}\backslash E$. From previous analysis, we can assume that there exists a constant $C_{1}>0$, such that $C_{1} \leq \rho(x) \leq r_{0}$. From (\ref{4.9}),  we know that $r_{\mathcal{R}}(x)$ is a function $\phi(x) = G(x,\nu(x))$ on $\bar{\Omega}\backslash E$, and $G(x,m)$ is continuous on $\bar{\Omega} \times \bar{D}$.
  
  To prove $r_{\mathcal{R}}(x)$ is continuous on $\bar{\Omega}\backslash E$, we only need to prove $r_{\mathcal{R}}(x)$ is both upper and lower semi-continuous on $\bar{\Omega}\backslash E$. We first prove  $r_{\mathcal{R}}(x)$ is upper semi-continuous on $\bar{\Omega}\backslash E$, that is, for any $\alpha \in \mathbb{R}$, the set $M_{\alpha} = \{x \in \bar{\Omega}\backslash E;~\phi(x)\leq \alpha\}$ is a closed set in $\bar{\Omega}\backslash E$. Then we need to prove that for a sequence $x_{k} \in M_{\alpha}$ and $x_{0} \in \bar{\Omega}\backslash E$, if $x_{k} \rightarrow x_{0}$, then $x_{0} \in M_{\alpha}$.

  We claim that for $x_{k},x_{0} \in \bar{\Omega}\backslash E$, if $x_{k} \rightarrow x_{0}$, then there exists a subsequence $x_{k_{j}}$, such that $\nu(x_{k_{j}})\rightarrow \nu(x_{0})$ as $j\rightarrow \infty$.
  
  Indeed, suppose that $\Gamma(P_{k},b_{k})$ support $\mathcal{R}$ at $\rho(x_{k})x_{k}$, then we have
  $$ \rho(x_{k}) = h(x_{k},P_{k},b_{k})  \quad \text{and} \quad  \rho(x) \geq h(x,P_{k},b_{k}) ~ \text{for all} ~ x \in \bar{\Omega}. $$
  Hence from Lemma \ref{lem3.8}, there exist subsequences $b_{k_{j}} \rightarrow b_{0}$ and $P_{k_{j}} \rightarrow P_{0}$ as $j \rightarrow \infty$. From Lemma \ref{lem3.9}, the claim holds true.
  
  Consequently, if $x_{k} \in M_{\alpha}$, then $\phi(x_{k}) = G(x_{k},\nu(x_{k})) \leq \alpha$. However, from claim, there exists a subsequence $x_{k_{j}}$, such that $\nu(x_{k_{j}})\rightarrow \nu(x_{0})$ as $j\rightarrow \infty$. Then since $G$ is continuous, we know that $r_{\mathcal{R}}(x)$ is upper semi-continuous on $\bar{\Omega}\backslash E$.

  Using the similar argument, we can prove that $r_{\mathcal{R}}(x)$ is lower semi-continuous on $\bar{\Omega}\backslash E$. Then $r_{\mathcal{R}}(x)$ is continuous on $\bar{\Omega}\backslash E$.
\end{proof}

\begin{remark}
  From Lemmas \ref{lem3.2} and \ref{lem3.4}, the singular points set of $\rho$ is a null set, then $r_{\mathcal{R}}(x)$ is well-defined on $\Omega$ almost everywhere. Hence $r_{\mathcal{R}}(x)$ is measurable in $\Omega$.
  \label{rem4.3}
\end{remark}

From above analysis, we can get the following lemma and theorem, which are useful in proving the existence of the weak solution.

\begin{lemma}
  Let $\mathcal{R}_{k}$ and $\mathcal{R}$ be refractors with defining functions $\rho_{k}(x)$ and $\rho(x)$, the corresponding Fresnel coefficients are $t_{k}$ and $t$. Suppose that $\rho_{k}\rightarrow \rho$ pointwise in $\bar{\Omega}$ and there exists a constant $C_{1}>0$, such that $C_{1} \leq \rho_{k}(x) \leq r_{0}$ in $\bar{\Omega}$. Then for $y \notin E$, there exists a subsequence $t_{k_{j}}(y)\rightarrow t(y)$ as $j\rightarrow \infty$, where $E$ is the union of singular points of refractors $\mathcal{R}_{k}$ and $\mathcal{R}$.
  \label{lem4.1}
\end{lemma}

\begin{proof}
  Given $y \notin E$, then there exist $b_{k}$ and $P_{k} \in \bar{D}$, such that 
  $$\rho_{k}(y) = h(x,P_{k},b_{k}) \quad \text{and} \quad \rho_{k}(z) \geq h(x,P_{k},b_{k}) ~\text{for all}~ z\in \bar{\Omega},$$
  and satisfies $x \cdot P_{k} \geq I(P_{k},b_{k})\vert P_{k} \vert.$ Since we have $C_{1} \leq \rho_{k}(x) \leq r_{0}$, then we have
  $$C_{1}(1 - \kappa) + \kappa \vert P_{k} \vert \leq b_{k} \leq \sqrt{\kappa^{2} \vert P_{k} \vert^{2} - (\kappa^{2} - 1)C_{1}^{2}}.$$
  So $b_{k}$s are bounded, then there exist subsequence $b_{k_{j}}\rightarrow b$ and $P_{k_{j}}\rightarrow P_{0} \in \bar{D}$. Hence $\Gamma(P_{0},b_{0})$ supports $\mathcal{R}$ at $y\rho(y)$, which implies $y \in \mathcal{T}_{\mathcal{R}}(m)$. For $y \notin E$, the normal $\nu_{k_{j}}(y)$ to the refracting oval $\Gamma(P_{k_{j}},b_{k_{j}})$ equals to the normal to the refractor $\mathcal{R}_{k}$ at $y$, and the normal $\nu(y)$ to the refracting oval $\Gamma(P,b)$ equals to the normal to the refractor $\mathcal{R}$ at $y$. Since $\Gamma(P_{k_{j}},b_{k_{j}}) \rightarrow \Gamma(P,b)$ as $j\rightarrow \infty$, then $\nu_{k_{j}}(y) \rightarrow \nu(y)$ for $y \notin E$ as  $j\rightarrow \infty$. So we have $t_{k_{j}}(y)\rightarrow t(y)$ as $j \rightarrow \infty$.
\end{proof}

\begin{theorem}
  Assume that the hypotheses and notations of Lemma \ref{lem4.1} hold, and let $F \subset \bar{D}$ be a compact set, set $F_{k} = \mathcal{T}_{\mathcal{R}_{k}}(F)$. Then for all $y \notin E$, we have

  \begin{align}
      & (a) \varlimsup_{k \rightarrow \infty} \chi_{F_{k}}(y)t_{k}(y) = t(y)\varlimsup_{k \rightarrow \infty} \chi_{F_{k}}(y),\label{4.12} \\
      & (b) \varliminf_{k \rightarrow \infty} \chi_{F_{k}}(y)t_{k}(y) = t(y)\varliminf_{k \rightarrow \infty} \chi_{F_{k}}(y),\label{4.13} \\
      & (c) \varlimsup_{k \rightarrow \infty} \chi_{F_{k}}(y)t_{k}(y) \leq \chi_{\mathcal{T}_{\mathcal{R}}(F)}(y)t(y),\label{4.14} \\
      & (d) \varlimsup_{k \rightarrow \infty}\int_{\bar{\Omega}}\chi_{F_{k}}(x)t_{k}(x)f(x)dx \leq \int_{\mathcal{T}_{\mathcal{R}}(F)}t(x)f(x)dx.\label{4.15}
  \end{align}
  \label{thm4.1}
\end{theorem}

\begin{proof}
  See Theorem 7.3 in \cite{St16}.
\end{proof}

\section{Definition of the weak solution}\label{Section 5}

\sloppy{}

In this section, we give the definition of the weak solution to the near field refraction problem with loss of energy in negative refractive index medium. We first give the definition of refractor measure originated from \cite{St17}, then the definition of the weak solution is given based on refractor measure.

\begin{definition}
  Suppose $\mathcal{R}$ is a refractor from $\bar{\Omega}$ to $\bar{D}$, $f \in L^{1}(\bar{\Omega})$ and $\inf\limits_{\bar{\Omega}}f >0$. The refractor measure associated with $\mathcal{R}$ and $f$ is defined by a set function on Borel subsets of $\bar{D}$:
  \begin{equation}\label{5.1}
    G_{\mathcal{R}}(F) := \int_{\mathcal{T}_{\mathcal{R}}(F)}f(x)t_{\mathcal{R}}(x)dx,
  \end{equation}
  where $dx$ is the surface measure on $S^{n-1}$.
  \label{def5.1}
\end{definition}

\begin{remark}
  $G_{\mathcal{R}}(F)$ is a finite Borel measure defined on $\mathcal{M}$, where $\mathcal{M}$ is defined in Lemma \ref{lem3.6} $(b)$.
  \label{rem5.1}
\end{remark}

Now we can define the weak solution to the near field refraction problem with loss of energy in negative refractive index medium.

\begin{definition}
  Suppose that $\mu$ is a Radon measure on the Borel subset of $\bar{D}$ and $f \in L^{1}(\bar{\Omega})$, a refractor $\mathcal{R}$ is a weak solution to the near field refraction problem with loss of energy with emitting illumination intensity $f(x)$ and prescribe refracted illumination intensity $\mu$ if for any Borel set $\omega \subset \bar{D}$, there holds:
  \begin{equation}\label{5.2}
    G_{\mathcal{R}}(\omega) = \int_{\mathcal{T}_{\mathcal{R}}(\omega)}f(x)t_{\mathcal{R}}(x)dx \geq \mu(\omega).
  \end{equation}
  \label{def5.2}
\end{definition}

\begin{remark}
  Since a small portion of energy is used for internal reflection, a little extra energy is required to ensure that light can be refracted into $\bar{D}$, so we use ``~$\geq$'' in (\ref{5.2}). 
\label{rem5.2}
\end{remark}


  

\section{Existence of the weak solution}\label{Section 6}

\sloppy{}

From \cite{GH14}, we know that if the weak solution to the near field refraction problem with loss of energy in negative refractive index medium exists, then $\rho$ satisfies the following inequality involving a Monge-Amp\`ere type operator:

\begin{equation}\label{6.1}
  \det (D^{2}\rho + \mathcal{A}) \leq \frac{f(x)t_{\mathcal{R}}(x)}{g(T(x)) \sqrt{1 - \vert x \vert^{2}}\vert D\psi \vert (2t)^{n -1}\rho^{n}(-\beta)F^{n-2}(F + D\rho \cdot D_{p}F)}.
\end{equation}
The definitions of $\mathcal{A}$, $\psi$, $\beta$ and $F$ see Appendix in \cite{GH14}.

In this section, we discuss the existence of the weak solution to the near field refraction problem with loss of energy in negative refractive index medium for the case $\kappa < -1$ and $-1 < \kappa <0$. For each case, we first give some lemmas which are crucial for proving the existence of the weak solution for the discrete case. Then we prove the existence of the weak solution under the assumption that $\mu$ equals finite sum of $\delta$-measures. Finally, the main results of this paper, that is, the existence of the weak solution when $\mu$ is a finite Radon measure is proved by applying the previous existence results when $\mu$ equals to finite sum of $\delta$-measures. At the end of this section, we also discuss the critical case $\kappa = -1$ briefly.

\subsection{Case $\kappa < -1$}\label{Section 6.1}

In this subsection, we discuss the existence of the weak solution for the case $\kappa < -1$. We first assume that $\mu$ equals finite sum of $\delta$-measures, hence all rays are refracted into finite directions. In order to establish the existence of the weak solution of the near field refraction problem with loss of energy for the case $\kappa <-1$ when $\mu$ is discrete measure, see Theorem \ref{thm6.1}, we first need some lemmas, see Lemmas \ref{lem6.1} $-$ \ref{lem6.3}.

\begin{remark}
  Suppose that $P_{1},P_{2},\ldots, P_{m},~m \geq 2$ are discrete points in $\bar{D}$, then for $\mathbf{b} = (b_{1},b_{2},\cdots ,b_{m}) \in \mathbb{R}^{m}$, the refractor is defined as
  \begin{equation}\label{6.2}
    \mathcal{R}(\mathbf{b}) = \{\rho(x)x;~x\in \bar{\Omega},~\rho(x)=\max_{1\leq j \leq m}h(x,P_{j},b_{j})\},
  \end{equation}
  where
  \begin{equation*}
    h(x,P_{j},b_{j}) = \frac{(\kappa^{2} x\cdot P_{j} - b_{j}) - \sqrt{(\kappa^{2} x\cdot P_{j} - b_{j})^{2} - (\kappa^{2} - 1)(\kappa^{2}\vert P_{j} \vert^{2} - b_{j}^{2})}}{\kappa^{2} - 1}.
  \end{equation*}
  \label{rem6.1}
\end{remark}

Now we show the existence of the weak solution when $\mu$ equals finite sum of $\delta$-measures.

\begin{theorem}
   Suppose that the assumptions $(A1)$ $-$ $(A5)$ hold. Let $\mu$ be the Borel measure defined on $\bar{D}$ by $\mu = \sum\limits_{i = 1}^{m}g_{i}\delta_{m_{i}}(\omega)$ with $\mu(\bar{D}) = \sum\limits_{j = 1}^{m} g_{j}$, where $\omega \subset \bar{D}$ is a Borel set. Then for any $b_{1}$ satisfies the condition $\kappa \vert P_{1} \vert + \alpha \leq b_{1} <\vert P_{1} \vert$, where $\alpha = -\kappa \sqrt{\dfrac{\kappa - 1}{\kappa + 1}}\sup\limits_{P \in \bar{D}} \vert P \vert$, there exists $(b_{2},\ldots,b_{m})$, such that the refracting oval $\mathcal{R} = \{\rho(x)x;~x\in \bar{\Omega}\}$ is the weak solution to the near field refraction problem with loss of energy for the case $\kappa < -1$, where $\rho(x) = \max\limits_{1 \leq j \leq m}h(x,P_{j},b_{j})$.
  \label{thm6.1}
\end{theorem}

In order to prove Theorem \ref{thm6.1}, we need the following Lemmas \ref{lem6.1} $-$ \ref{lem6.3}.

\begin{lemma}
  Under the assumptions in Theorem \ref{thm6.1}, suppose also that assumptions $(A1)$ and $(A2)$ hold. Define a set $W := \left\{(b_{2},\ldots,b_{m});~\kappa \vert P_{j} \vert < b_{j} \leq (\tau + \dfrac{1}{\kappa})\vert P_{j} \vert,~j = 2,\ldots,m\right\}$, and for any $\mathbf{b} \in W$, we have $\displaystyle\int_{\mathcal{T}_{\mathcal{R}(\mathbf{b})}(P_{j})}f(x)t_{\mathcal{R}(\mathbf{b})}(x)dx \leq g_{i}, 2 \leq i \leq m$, then
  \newline (a) $W \neq \emptyset$;
  \newline (b) If $\mathbf{b} = (b_{2}, \ldots ,b_{m}) \in W$, then for $2 \leq j \leq m$, we have 
  \begin{equation*}
    b_{j} \geq \kappa \vert P_{j} \vert - \sqrt{\frac{\kappa - 1}{\kappa + 1}(\kappa^{2} \vert P_{1} \vert^{2} - b_{1}^{2})}.
  \end{equation*}
  \label{lem6.1}
\end{lemma}

\begin{proof}
  $(a)$ For any $b_{1}$ satisfies the condition $\kappa \vert P_{1} \vert + \alpha \leq  b_{1} <\vert P_{1} \vert$, where $\alpha = -\kappa \sqrt{\dfrac{\kappa - 1}{\kappa + 1}}\sup\limits_{P \in \bar{D}} \vert P \vert$, we claim that $\mathcal{R}(\mathbf{b}) = \Gamma(P_{1},b_{1}) = \{h(x,P_{1},b_{1})x;~x \in S^{n-1}, x\cdot P_{1} \geq b_{1}\}$.
  
  Indeed, from Lemma \ref{lem3.1}, we have
  \begin{equation*}
    \begin{aligned}
    h(x,P_{1},b_{1}) & \geq \frac{\kappa \vert P_{1} \vert - b_{1}}{\kappa - 1} \geq -\frac{\kappa}{\kappa - 1}\sqrt{\frac{\kappa - 1}{\kappa + 1}} \sup\limits_{P \in \bar{D}} \vert P \vert \\
    & \geq \frac{1}{\sqrt{\kappa^{2} - 1}}\sqrt{\kappa^{2} \vert P_{j} \vert^{2} - b_{j}^{2}} \geq h(x,P_{j},b_{j}).
    \end{aligned}
  \end{equation*}
  Then for $j \neq 1$, we have $G_{\mathcal{R}(\mathbf{b})}(\{P_{j}\}) = 0 < g_{j}$. Hence $W \neq \emptyset$.
  
  $(b)$ For $\mathcal{R}(\mathbf{b})$ is a near field refractor, then there exists $x_{0} \in \bar{\Omega}$, such that $\Gamma(P_{1},b_{1})$ supports $\mathcal{R}(\mathbf{b})$ at $\rho(x_{0})x_{0}$. Then we have
  \begin{equation*}
    \frac{\sqrt{\kappa^{2} \vert P_{1} \vert^{2} - b_{1}^{2}}}{\sqrt{\kappa^{2} - 1}} \geq \rho(x_{0}) \geq h(x_{0},P_{j},b_{j}) \geq \frac{\kappa \vert P_{j} \vert - b_{j}}{\kappa - 1}.
  \end{equation*}
  Hence we have
  \begin{equation*}
    b_{j} \geq \kappa \vert P_{j} \vert - \sqrt{\frac{\kappa - 1}{\kappa + 1}(\kappa^{2} \vert P_{1} \vert^{2} - b_{1}^{2})}.
  \end{equation*}
\end{proof}

\begin{lemma}
  Suppose that $\mathbf{b}_{k} = (b_{2}^{k},\ldots,b_{m}^{k}) \in W$ and $\mathbf{b}_{0} = (b_{2}^{0},\ldots,b_{m}^{0})$, such that $\mathbf{b}_{k} \rightarrow \mathbf{b}_{0}$ as $k \rightarrow \infty$. Let $\mathcal{R}_{k} = \mathcal{R}(\mathbf{b}_{k}) = \{\rho_{k}(x)x;~x \in \bar{\Omega}\}$ and $\mathcal{R}_{0} = \mathcal{R}(\mathbf{b}_{0}) = \{\rho(x)x;~x \in \bar{\Omega}\}$, then $\rho_{k} \rightarrow \rho$ uniformly on $\bar{\Omega}$.
  \label{lem6.2}
\end{lemma}

\begin{proof}
  From the proof of Lemma \ref{lem6.1} $(b)$, we have
  $$\rho_{k}(x) \leq \frac{\sqrt{\kappa^{2} \vert P_{1} \vert^{2} - b_{1}^{2}}}{\sqrt{\kappa^{2} - 1}}.$$
  Multiplying by $\sqrt{\dfrac{\kappa^{2} + 1}{\kappa^{2} - 1}}$, we have
  \begin{equation*}
    \frac{\sqrt{\kappa^{2} + 1}}{\kappa^{2} - 1}\frac{\sqrt{\kappa^{2} \vert P_{1} \vert^{2} - b_{1}^{2}}}{\sqrt{\kappa^{2} - 1}} \geq \frac{\sqrt{\kappa^{2} + 1}}{\kappa^{2} - 1} \rho_{k}(x) \geq \rho_{k}(x) = \max\limits_{1 \leq j \leq m}h(x,P_{j},b_{j}^{k}) \geq C_{1},
  \end{equation*}
  where $C_{1}$ is defined in Lemma \ref{lem3.8} and $b_{1}^{k} = b_{1}$. So we have $\rho_{k}(x) \rightarrow \rho(x)$ uniformly on $\bar{\Omega}$.
\end{proof}

\begin{lemma}
  Suppose that $\delta \geq \kappa \vert P_{j} \vert - \sqrt{\dfrac{\kappa - 1}{\kappa + 1}(\kappa^{2} \vert P_{1} \vert^{2} - b_{1}^{2})}$, then $G_{\mathcal{R}(\mathbf{b})}(\{P_{j}\}) = \displaystyle\int_{\mathcal{T}_{\mathcal{R}(\mathbf{b})}(P_{j})}f(x)t_{\mathcal{R}}(x)dx$ is continuous on the region $R_{\delta} = \{(b_{2},\ldots,b_{m});~b_{j} \geq \delta,j=2,\ldots,m\}$.
  \label{lem6.3}
\end{lemma}

\begin{proof}
  Suppose that $\mathbf{b}_{k} = (b_{1}^{k},\ldots,b_{m}^{k})$ is a sequence converges to $\mathbf{b}_{0} = (b_{1}^{0},\ldots,b_{m}^{0})$ in $R_{\delta}$, and let $\mathcal{R}(\mathbf{b}_{k}) = \{\rho_{k}(x)x;~x\in \bar{\Omega}\}$, $\mathcal{R}(\mathbf{b}_{0}) = \{\rho_{0}(x)x;~x\in \bar{\Omega}\}$. From Lemma \ref{lem6.2}, $\rho_{k}\rightarrow \rho$ uniformly on $\bar{\Omega}$. Then for any $x \in \bar{\Omega}$, we have
  \begin{equation*}
    C_{1} \leq \rho_{k}(x) \leq \frac{\sqrt{\kappa^{2} + 1}}{\kappa^{2} - 1}\frac{\sqrt{\kappa^{2} \vert P_{1} \vert^{2} - b_{1}^{2}}}{\sqrt{\kappa^{2} - 1}}.
  \end{equation*}
  
  Suppose that $G \subset \bar{D}$ is a neighborhood of $P_{j}$, such that $P_{k} \notin G_{j}$ for $k \neq j$. If $x_{0} \in \mathcal{T}_{\mathcal{R}(\mathbf{b}_{k})}(G)$ and $x_{0} \notin E$, then there exist $P \in G$ and $b$, such that 
  $$\rho_{k}(x_{0}) = h(x_{0},P,b) \quad \text{and} \quad \rho_{k}(x) \geq h(x,P,b)~\text{for any} ~ x \in \bar{\Omega},$$
  and we have $x \cdot P \geq I(P,b)\vert P \vert$. From Lemma \ref{lem3.9}, we have $\mathcal{T}_{\mathcal{R}(\mathbf{b}_{k})}(G) \subset \mathcal{T}_{\mathcal{R}(\mathbf{b}_{k})}(P_{j}) \cup E$, hence
  \begin{equation}\label{6.3}
    \begin{aligned}
    \int_{\mathcal{T}_{\mathcal{R}(\mathbf{b}_{0})}(G)}f(x)t_{\mathcal{R}(\mathbf{b}_{0})}(x)dx & \leq \int_{\varliminf\limits_{k \rightarrow \infty}\mathcal{T}_{\mathcal{R}(\mathbf{b}_{k})}(P_{j}) \cup E}f(x)t_{\mathcal{R}(\mathbf{b}_{0})}(x)dx \\
    & \leq \int_{\bar{\Omega}}\chi_{\varliminf\limits_{k \rightarrow \infty}\mathcal{T}_{\mathcal{R}(\mathbf{b}_{k})}(P_{j})}(x)f(x)t_{\mathcal{R}(\mathbf{b}_{0})}(x)dx.
    \end{aligned}
  \end{equation}
  From Theorem \ref{thm4.1} $(b)$, we have
  \begin{equation}\label{6.4}
    \int_{\mathcal{T}_{\mathcal{R}(\mathbf{b}_{0})}(G)}f(x)t_{\mathcal{R}}(x)dx \leq \int_{\bar{\Omega}}\chi_{\varliminf\limits_{k \rightarrow \infty}\mathcal{T}_{\mathcal{R}(\mathbf{b}_{k})}(P_{j})}(x)f(x)t_{\mathcal{R}(\mathbf{b}_{k})}(x)dx.
  \end{equation}
  For
  \begin{equation}\label{6.5}
    \chi_{\varliminf\limits_{k \rightarrow \infty}\mathcal{T}_{\mathcal{R}(\mathbf{b}_{k})}(P_{j})}(x) = \varliminf\limits_{k \rightarrow \infty}\chi_{\mathcal{T}_{\mathcal{R}(\mathbf{b}_{k})}(P_{j})}(x),
  \end{equation}
  then applying (\ref{6.5}) to (\ref{6.4}) and using reverse Fatou Lemma, we have
  \begin{equation}\label{6.6}
    \begin{aligned}
    \int_{\mathcal{T}_{\mathcal{R}(\mathbf{b}_{0})}(G)}f(x)t_{\mathcal{R}_{\mathbf{b}_{0}}}(x)dx & \leq \varliminf\limits_{k \rightarrow \infty}\int_{\bar{\Omega}}\chi_{\mathcal{T}_{\mathcal{R}(\mathbf{b}_{k})}(P_{j})}(x)t_{\mathcal{R}(\mathbf{b}_{k})}(x)f(x)dx \\
    & = \varliminf\limits_{k \rightarrow \infty} \int_{\mathcal{T}_{\mathcal{R}(\mathbf{b}_{k})}(P_{j})}t_{\mathcal{R}(\mathbf{b}_{k})}(x)f(x)dx.
    \end{aligned}
  \end{equation}
  Besides, we also have
  \begin{equation}\label{6.7}
    \chi_{\varlimsup\limits_{k \rightarrow \infty} \mathcal{T}_{\mathcal{R}({\mathbf{b}_{k}})}(P_{j})}(x) = \varlimsup\limits_{k \rightarrow \infty} \chi_{\mathcal{T}_{\mathcal{R}({\mathbf{b}_{k}})}(P_{j})}(x).
  \end{equation}
  From reverse Fatou Lemma, Lemma \ref{lem3.9}, Theorem \ref{thm4.1} and (\ref{6.7}), we have
  \begin{equation}\label{6.8}
    \begin{aligned}
    \varlimsup\limits_{k \rightarrow \infty}\int_{\mathcal{T}_{\mathcal{R}({\mathbf{b}_{k}})}(P_{j})}t_{\mathcal{R}(\mathbf{b}_{k})}(x)f(x)dx & = \varlimsup\limits_{k \rightarrow \infty}\int_{\bar{\Omega}}\chi_{\mathcal{T}_{\mathcal{R}({\mathbf{b}_{k}})}(P_{j})}(x)t_{\mathcal{R}(\mathbf{b}_{k})}(x)f(x)dx \\
    & = \int_{\bar{\Omega}}\chi_{\varlimsup\limits_{k \rightarrow \infty} \mathcal{T}_{\mathcal{R}({\mathbf{b}_{k}})}(P_{j})}(x)f(x)t_{\mathcal{R}(\mathbf{b}_{0})}(x)dx \\
    & = \int_{\varlimsup\limits_{k \rightarrow \infty}\mathcal{T}_{\mathcal{R}({\mathbf{b}_{k}})}(P_{j})}f(x)t_{\mathcal{R}(\mathbf{b}_{0})}(x)dx \\
    & \leq \int_{\mathcal{T}_{\mathcal{R}(\mathbf{b}_{0})}(G)}f(x)t_{\mathcal{R}(\mathbf{b}_{0})}(x)dx.
    \end{aligned}
  \end{equation}
  Combining (\ref{6.6}) with (\ref{6.8}), we obtain $G_{\mathcal{R}(\mathbf{b})}(P_{j})$ is continuous on the region $R_{\delta}$.
\end{proof}

Based on the above lemmas, now we prove the existence of the weak solution for the case $\kappa < -1$ when $\mu$ equals finite sum of $\delta$-measures.

\begin{proof}[Proof of Theorem \ref{thm6.1}]
  In order to prove that the refracting oval $\mathcal{R} = \{\rho(x)x;~x\in \bar{\Omega}\}$ is the weak solution to the near field refraction problem with loss of energy for the case $\kappa < -1$, we need to verify the following two conditions:
  
  (a) For any Borel set $F \subset \bar{D}$, there holds $\displaystyle\int_{\mathcal{T}_{\mathcal{R}}(F)}f(x)t_{\mathcal{R}}(x)dx \geq \mu(F)$;
  
  (b) For each Borel set with $P_{1} \notin F$, there holds $\displaystyle\int_{\mathcal{T}_{\mathcal{R}}(F)}f(x)t_{\mathcal{R}}(x)dx = \mu(F) = \sum\limits_{j=2}^{m}g_{j}$.

  We consider the mapping $d:W \rightarrow \mathbb{R}:~(b_{2},\ldots,b_{m}) \mapsto b_{2} + \ldots + b_{m}$. From Lemmas \ref{lem6.1} $-$ \ref{lem6.3}, we know that $W$ is a compact set, hence $d$ can attain its maximum at some points in $W$. Suppose that $d$ attains its maximum at $(a_{2},\ldots,a_{m})$, we claim that $\mathcal{R}(\mathbf{a})$ is the weak solution of the near field refraction problem, where $\mathbf{a} = (b_{1},a_{2},\ldots,a_{m})$.
  
  We first prove that $\displaystyle\int_{\mathcal{T}_{\mathcal{R}(\mathbf{a})}(P_{j})}f(x)t_{\mathcal{R}}(x)dx = g_{j}$ for $j = 2,\ldots,m$.
  
  Indeed, without loss of generality, we assume for contradiction that $\displaystyle\int_{\mathcal{T}_{\mathcal{R}(\mathbf{a})}(P_{2})}f(x)t_{\mathcal{R}}(x)dx < g_{2}$. Let $\varepsilon > 0$ and define $\bar{\mathbf{a}} = (b_{1},a_{2}-\varepsilon,\ldots,a_{m})$, the corresponding refracting oval is denoted as $\mathcal{R}(\bar{\mathbf{a}})$. From the continuity of $G_{\mathcal{R}}$, there holds $G_{\mathcal{R}(\bar{\mathbf{a}})}(P_{2}) = \displaystyle\int_{\mathcal{T}_{\mathcal{R}(\mathbf{a})}(P_{2})}f(x)t_{\mathcal{R}}(x)dx < g_{2}$ for $\varepsilon$ is sufficiently small, and for $j \neq 1,2$, we have $\mathcal{T}_{\mathcal{R}(\bar{\mathbf{a}})}(P_{j}) \subset \mathcal{T}_{\mathcal{R}(\mathbf{a})}(P_{j})$ almost everywhere, hence $\bar{\mathbf{a}} \in W$. This is a  contradiction with $(a_{2},\ldots,a_{m})$ is the maximum of $d$ in $W$, so we have $\displaystyle\int_{\mathcal{T}_{\mathcal{R}(\mathbf{a})}(P_{j})}f(x)t_{\mathcal{R}}(x)dx = g_{j}$ for $j = 2,\ldots,m$.
  
  Then we prove that $\displaystyle\int_{\mathcal{T}_{\mathcal{R}(\mathbf{a})}(P_{1})}f(x)t_{\mathcal{R}}(x)dx > g_{1}$.
  
  Indeed, we first prove that $\displaystyle\int_{\mathcal{T}_{\mathcal{R}(\mathbf{a})}(P_{1})}f(x)t_{\mathcal{R}}(x)dx \geq g_{1}$. 
  
  Let $b \in W$, then we have
  \begin{equation*}
    \begin{aligned}
    \sum_{j = 1}^{m}\int_{\mathcal{T}_{\mathcal{R}(\mathbf{b})}(P_{j})}f(x)t_{\mathcal{R}}(x)dx &= \int_{\bigcup\limits_{j=1}^{m}\mathcal{T}_{\mathcal{R}(\mathbf{b})}(P_{j})}f(x)t_{\mathcal{R}}(x)dx \\
    & = \int_{\bar{\Omega}}f(x)t_{\mathcal{R}}(x)dx \\
    & \geq (1 - C_{\varepsilon})\int_{\bar{\Omega}}f(x)dx \geq \mu(\bar{D}) = \sum_{j = 1}^{m}g_{j}.
    \end{aligned}
  \end{equation*}
  So we have
  \begin{equation*}
    (g_{1} - \int_{\mathcal{T}_{\mathcal{R}(b_{1})}(P_{1})}f(x)t_{\mathcal{R}}(x)dx) + \sum_{j = 2}^{m}(g_{j} - \int_{\mathcal{T}_{\mathcal{R}(b_{j})}(P_{j})}f(x)t_{\mathcal{R}}(x)dx) \leq 0.
  \end{equation*}
  For $b \in W$, then we have $\sum\limits_{j = 2}^{m}(g_{j} - \displaystyle\int_{\mathcal{T}_{\mathcal{R}(b_{j})}(P_{j})}f(x)t_{\mathcal{R}}(x)dx) \geq 0$. Hence $g_{1} \leq \displaystyle\int_{\mathcal{T}_{\mathcal{R}(b_{1})}(P_{1})}f(x)t_{\mathcal{R}}(x)dx$. Taking $b = a_{1}$, we have $\displaystyle\int_{\mathcal{T}_{\mathcal{R}(\mathbf{a})}(P_{1})}f(x)t_{\mathcal{R}}(x)dx \geq g_{1}$.
  
  Now we turn to prove $\displaystyle\int_{\mathcal{T}_{\mathcal{R}(\mathbf{a})}(P_{1})}f(x)t_{\mathcal{R}}(x)dx > g_{1}$.
  
  Suppose by the contradiction that $\displaystyle\int_{\mathcal{T}_{\mathcal{R}(\mathbf{a})}(P_{1})}f(x)t_{\mathcal{R}}(x)dx = g_{1}$, then we have 
  \begin{equation*}
  \begin{aligned}
    \int_{\bar{\Omega}}f(x)t_{\mathcal{R}}(x)dx & = \sum_{j=1}^{m}\int_{\mathcal{T}_{\mathcal{R}(\mathbf{a})}(P_{j})}f(x)t_{\mathcal{R}}(x)dx \\
    & = \sum_{j=1}^{m}g_{j} \leq (1 - C_{\varepsilon})\int_{\bar{\Omega}}f(x)dx,
  \end{aligned}
  \end{equation*}
  namely
  $$\displaystyle\int_{\bar{\Omega}}f(x)[(1 - C_{\varepsilon}) - t_{\mathcal{R}}(x)]dx \geq 0.$$
  From the definition of $t_{\mathcal{R}}(x)$, we must have $t_{\mathcal{R}}(x) \geq 1 - C_{\varepsilon}$. Since $\inf\limits_{\bar{\Omega}}f > 0$, then we must have $t_{\mathcal{R}}(x) = 1 - C_{\varepsilon}$ for a.e. $x \in \bar{\Omega}$. Then for $x \in \mathcal{T}_{\mathcal{R}(\mathbf{a})(P_{1})} \backslash E$, we have $r_{\mathcal{R}}(x) = C_{\varepsilon}$.
  
  We claim that the set $D = \left\{x \cdot \dfrac{P_{1}}{\vert P_{1} \vert};~x\in \mathcal{T}_{\mathcal{R}(\mathbf{a})}(P_{1}) \backslash E\right\}$ is infinite.
  
  Indeed, if not, then there exist $c_{1},\ldots,c_{n}$, such that $D=\{c_{1},\ldots,c_{n}\}$. Let $D_{j} = \left\{x \in \mathcal{T}_{\mathcal{R}(\mathbf{a})}(P_{1}) \backslash E;~x \cdot \dfrac{P_{1}}{\vert P_{1} \vert} = c_{j}\right\}$, then $\vert D_{j} \vert = 0$, hence $\vert D \vert = 0$. This is a contradiction with $\vert \mathcal{T}_{\mathcal{R}(\mathbf{a})}(P_{1}) \backslash E \vert > 0$. Hence the set $D$ is infinite. Besides, from Proposition \ref{prop4.1}, the set $\{s;~\psi(s) = c\}$ is finite for any constant $c$, then $\psi = C_{\varepsilon}$ cannot appears on $D$. So we must have $\displaystyle\int_{\mathcal{T}_{\mathcal{R}(\mathbf{a})}(P_{1})}f(x)t_{\mathcal{R}}(x)dx > g_{1}$.
  
  Finally, for the Borel set $F \subset D$, we have 
  $$\mathcal{T}_{\mathcal{R}}(F) = \bigcup\limits_{j=1}^{m}\mathcal{T}_{\mathcal{R}}(F) \cap \mathcal{T}_{\mathcal{R}}(P_{j}) = \bigcup\limits_{j=1}^{m}\mathcal{T}_{\mathcal{R}}(F \cap P_{j}),$$
  so we have
  \begin{equation*}
    \begin{aligned}
    \int_{\mathcal{T}_{\mathcal{R}}(F)}f(x)t_{\mathcal{R}}(x)dx & = \sum_{j=1}^{m}\int_{\mathcal{T}_{\mathcal{R}}(F \cap P_{j})}f(x)t_{\mathcal{R}}(x)dx \\
    & = \sum_{j;P_{j} \in F}\int_{\mathcal{T}_{\mathcal{R}}(F \cap P_{j})}f(x)t_{\mathcal{R}}(x)dx \\
    & = \sum_{j;P_{j} \in F}\int_{\mathcal{T}_{\mathcal{R}}(P_{j})}f(x)t_{\mathcal{R}}(x)dx \geq \mu(F).
    \end{aligned}
  \end{equation*}
  Then if $P_{1} \notin F$, we have $\displaystyle\int_{\mathcal{T}_{\mathcal{R}}(F)}f(x)t_{\mathcal{R}}(x)dx = \mu(F) = \sum\limits_{j=2}^{m}g_{j}$.
\end{proof}

Based on Theorem \ref{thm6.1}, now we can prove the existence of the weak solution to the near field refraction problem with loss of energy for the case $\kappa < -1$.

\begin{proof}[Proof of Theorem \ref{thm1.1}]
  In order to prove the existence of the weak solution for the case $\kappa < -1$, we need to show that given $P_{0} \in supp(\mu)$, then for any Borel set $F \subset D$ with $P_{0} \notin F$, there exists a refractor $\mathcal{R}$, such that $\displaystyle\int_{\mathcal{T}_{\mathcal{R}}(F)}f(x)t_{\mathcal{R}}(x)dx \geq \mu(F)$, where $\mu$ is the Radon measure defined on $\bar{D}$.

  We fix $\kappa \vert P_{1} \vert + \alpha \leq b_{1} < \vert P_{1} \vert$, and divide $D$ into a finite, disjoint union of Borel set $\{E_{j}\}$ with nonempty interiors. For each $E_{j}$, we have $diam(E_{j}) \leq \delta$, where $\delta$ is sufficient small. Suppose that $P_{0}$ belongs to the interior of some $E_{k}$. For $P_{0} \in supp(\mu)$, then $\mu(E_{k}) > 0$. We discard the sets in $E_{j}$ which makes $\mu(E_{j}) = 0$, and label the remaining sets as $F_{1}^{1},\ldots,F_{m_{1}}^{1}$. Without loss of generality, we assume that $P_{0} \in (F_{1}^{1})^{\circ}$ and $\mu(F_{j}^{1}) > 0$. Taking $P_{j}^{1} \in F_{j}^{1}$, then we have $P_{1}^{1} = P_{0}$.
  
  Defining a measure on $\bar{D}$ as $\mu_{1} = \sum\limits_{j=1}^{m_{1}}\mu(F_{j}^{1})\delta_{P_{j}^{1}}$, then we have
  $$\mu_{1}(\bar{D}) = \mu(\bar{D}) \leq (1 - C_{\varepsilon})\displaystyle\int_{\bar{\Omega}}f(x)dx.$$
  From Theorem \ref{thm6.1}, there exists a near field refractor $\mathcal{R}_{1} = \{\rho_{1}(x)x;~\rho_{1}(x) = \max\limits_{1 \leq j \leq m_{1}}h(x,P_{j}^{1},b_{j})\}$, such that $\mu_{1}(F) \leq \displaystyle\int_{\mathcal{T}_{\mathcal{R}_{1}}(F)}f(x)t_{\mathcal{R}_{1}}(x)dx$, and the equality holds when $P_{0} \notin F$.
  
  Next, we divide each $F_{j}^{1}$ into finite, disjoint union of Borel sets $\{E_{j}^{'}\}$ with nonempty interiors. For each $E_{j}^{'}$, we have $diam(E_{j}^{'}) \leq \dfrac{\delta}{2}$, where $\delta$ is sufficient small. Suppose that $P_{0}$ belongs to the interior of some $E_{k}^{'}$. For $P_{0} \in supp(\mu)$, then $\mu(E_{k}^{'}) > 0$. We discard the sets in $E_{j}^{'}$ which makes $\mu(E_{j}^{'}) = 0$, and label the remaining sets as $F_{1}^{2},\ldots,F_{m_{2}}^{2}$. Without loss of generality, we assume that $F_{1}^{2} \subset F_{1}^{1}$ and $P_{0} \in (F_{1}^{2})^{\circ}$. Taking $P_{j}^{2} \in F_{j}^{2}$, then we have $P_{1}^{2} = P_{0}$.
  
  Defining a measure on $\bar{D}$ as $\mu_{2} = \sum\limits_{j=1}^{m_{2}}\mu(F_{j}^{2})\delta_{P_{j}^{2}}$, then we have
  $$\mu_{2}(\bar{D}) = \mu(\bar{D}) \leq (1 - C_{\varepsilon})\displaystyle\int_{\bar{\Omega}}f(x)dx.$$
  From Theorem \ref{thm6.1}, there exists a near field refractor $\mathcal{R}_{2} = \{\rho_{2}(x)x;~\rho_{2}(x) = \max\limits_{1 \leq j \leq m_{2}}h(x,P_{j}^{2},b_{j})\}$, such that $\mu_{2}(F) \leq \displaystyle\int_{\mathcal{T}_{\mathcal{R}_{2}}(F)}f(x)t_{\mathcal{R}_{2}}(x)dx$, and the equality holds when $P_{0} \notin F$.
  
  Continuing in this way, we obtain a sequence of disjoint Borel sets $\{F_{j}^{l}\}$, where $1 \leq j \leq m_{l}$ and $diam(F_{j}^{l}) \leq \dfrac{\delta}{2^{l}}$. For any $F_{j}^{l}$, we have $\mu(F_{j}^{l}) > 0$ with $P_{0} \in (F_{j}^{l})^{\circ}$ and $F_{1}^{l + 1} \subset F_{1}^{l}$, we take $P_{j}^{l} \in F_{j}^{l}$ such that $P_{1}^{l} = P_{0}$ holds for any $j$ and $l$. Then the measure $\mu_{l} = \sum\limits_{j=1}^{m_{l}}\mu(F_{j}^{l})\delta_{P_{j}^{l}}$ satisfies 
   $$\mu_{l}(\bar{D}) = \mu(\bar{D}) \leq (1 - C_{\varepsilon})\displaystyle\int_{\bar{\Omega}}f(x)dx.$$
  From Theorem \ref{thm6.1}, there exists a near field refractor $\mathcal{R}_{l} = \{\rho_{l}(x)x;~\rho_{l}(x) = \max\limits_{1 \leq j \leq m_{l}}h(x,P_{j}^{l},b_{j})\}$, such that $\mu_{l}(F) \leq \displaystyle\int_{\mathcal{T}_{\mathcal{R}_{l}}(F)}f(x)t_{\mathcal{R}_{l}}(x)dx$, and the equality holds when $P_{0} \notin F$.
  
  Let $x_{0},x_{1} \in \bar{\Omega}$ and suppose that $\Gamma(P_{0},b_{0})$ supports $\mathcal{R}_{l}$ at $\rho(x_{0})x_{0}$. From Lemma \ref{lem3.2}, $\rho_{l}$ is Lipschitz continuous. From the proof of Theorem 3.6 in \cite{St17}, $\rho_{l}$ is uniformly bounded. Hence $\{\rho_{l};~l \geq 1\}$ is a family of  uniformly bounded and equicontinuous functions. Then from  Arezlà-Ascoli Theorem, $\rho_{l} \rightarrow \rho$ uniformly on $\bar{\Omega}$.
  
  Then we need to prove that $\mathcal{R} = \{\rho(x)x;~x\in \bar{\Omega}\}$ is a near field refractor.
  
  Indeed, fix $x_{0} \in \bar{\Omega}$, then there exist $P_{l}$ and $b_{l}$, such that $\Gamma(P_{l},b_{l})$ supports $\mathcal{R}_{l}$ at $\rho(x_{0})x_{0}$, namely
  $$\rho_{l}(x_{0}) = h(x_{0},P_{l},b_{l}) \quad \text{and} \quad \rho_{l}(x) \geq h(x,P_{l},b_{l}) \quad \text{for all}~x \in \bar{\Omega}.$$
  From Lemma \ref{lem3.8}, we have
  \begin{equation*}
     C_{1}(1 - \kappa) + \kappa \vert P_{l} \vert \leq b_{l} \leq \sqrt{\kappa^{2} (\vert P_{l} \vert^{2} - C_{1}^{2}) + C_{1}^{2}},
  \end{equation*}
  then there exist $P_{0}$ and $b_{0}$, such that $P_{l} \rightarrow P_{0}$ and $b_{l} \rightarrow b_{0}$ as $l \rightarrow \infty$. 
  
  Now we need to verify that $\Gamma(P_{0},b_{0})$ supports $\mathcal{R}$ at $\rho(x_{0})x_{0}$. For we have
  \begin{equation*}
    \rho(x_{0}) = \lim\limits_{l\rightarrow \infty}\rho_{l}(x_{0}) = \lim\limits_{l\rightarrow \infty}h(x_{0},P_{l},b_{l}) = h(x_{0},P_{0},b_{0})
  \end{equation*}
  and
  $$\rho(x) = \lim\limits_{l\rightarrow \infty}\rho_{l}(x) \geq \lim\limits_{l\rightarrow \infty}h(x,P_{l},b_{l}) = h(x,P_{0},b_{0}) \quad \text{for any}~ x \in \bar{\Omega}.$$
  Hence $\Gamma(P_{0},b_{0})$ supports $\mathcal{R}$ at $\rho(x_{0})x_{0}$.
  
  Finally, we need to prove that $G_{\mathcal{R}_{l}}(F) = \displaystyle\int_{\mathcal{T}_{\mathcal{R}_{l}}(F)}f(x)t_{\mathcal{R}_{l}}(x)dx$ converges to $G_{\mathcal{R}}(F) = \displaystyle\int_{\mathcal{T}_{\mathcal{R}}(F)}f(x)t_{\mathcal{R}}(x)dx$ weakly as $l \rightarrow \infty$. Namely, $\varlimsup\limits_{l \rightarrow \infty}G_{\mathcal{R}_{l}}(F) \leq G_{\mathcal{R}}(F)$ holds for any compact sets $F \subset \bar{D}$ and $G_{\mathcal{R}}(U) \leq \varliminf\limits_{l \rightarrow \infty}G_{\mathcal{R}_{l}}(U)$ holds for any open sets $U \subset \bar{D}$.
  
  Indeed, from Theorem \ref{thm4.1}, for any compact set $F$, we have
  \begin{equation*}
    \begin{aligned}
    \varlimsup_{l \rightarrow \infty}G_{\mathcal{R}_{l}}(F) & = \varlimsup_{l \rightarrow \infty}\int_{\mathcal{T}_{\mathcal{R}_{l}}(F)}f(x)t_{\mathcal{R}_{l}}(x)dx \\
    & \leq \int_{\bar{\Omega}}\varlimsup_{l \rightarrow \infty}\chi_{\mathcal{T}_{\mathcal{R}_{l}}(F)}(x)f(x)t_{\mathcal{R}_{l}}(x)dx \\
    & \leq \int_{\mathcal{T}_{\mathcal{R}}(F)}f(x)t_{\mathcal{R}}(x)dx = G_{\mathcal{R}}(F).
    \end{aligned}
  \end{equation*}
  For any open set $U$, we have
  \begin{equation*}
    \begin{aligned}
    G_{\mathcal{R}}(U) & = \int_{\mathcal{T}_{\mathcal{R}}(U)}f(x)t_{\mathcal{R}}(x)dx \\
    & \leq \int_{\bar{\Omega}}\varliminf_{l \rightarrow \infty}\chi_{\mathcal{T}_{\mathcal{R}_{l}}(U)}(x)f(x)t_{\mathcal{R}_{l}}(x)dx \\
    & \leq \varliminf_{l \rightarrow \infty}\int_{\bar{\Omega}}\chi_{\mathcal{T}_{\mathcal{R}_{l}}(U)}(x)f(x)t_{\mathcal{R}_{l}}(x)dx \\
    & = \varliminf_{l \rightarrow \infty}G_{\mathcal{R}_{l}}(U).
  \end{aligned}
  \end{equation*}
  So we have $G_{\mathcal{R}_{l}}(F) \rightarrow G_{\mathcal{R}}(F)$ weakly as $l \rightarrow \infty$.
  
  Since $\mu_{l}(F) = G_{\mathcal{R}_{l}}(F)$ holds for any $F$ satisfies $P_{0} \notin F$, then $\mu_{l} \rightarrow G_{\mathcal{R}}$ weakly as $l \rightarrow \infty$ on $\bar{D} \backslash \{P_{0}\}$. We also have $\mu_{l} \rightarrow \mu$ weakly as $l \rightarrow \infty$, hence $\mu(F) = G_{\mathcal{R}}(F)$ holds for $P_{0} \notin F$. For any Borel set $E \subset \bar{D}$, we have $\mu_{l}(E) \leq G_{\mathcal{R}_{l}}(E)$. Hence $\mu(E) \leq G_{\mathcal{R}}(E)$.  
\end{proof}

\subsection{Case $-1 < \kappa < 0$}\label{Section 6.2}

In this subsection, we discuss the existence of the weak solution for the case $-1< \kappa < 0$. Similar to Section \ref{Section 6.1}, we first assume that $\mu$ equals finite sum of $\delta$-measures, hence all rays are refracted into finite directions. In order to establish the existence of the weak solution to the near field refraction problem with loss of energy for the case $-1< \kappa < 0$ when $\mu$ is discrete measure, see Theorem \ref{thm6.2}, we also need some lemmas, see Lemmas \ref{lem6.4} $-$ \ref{lem6.6}.

\begin{remark}
  Suppose that $P_{1},P_{2},\ldots, P_{m},~m \geq 2$ are discrete points in $\bar{D}$, then for $\mathbf{b} = (b_{1},b_{2},\cdots ,b_{m}) \in \mathbb{R}^{m}$, the refractor is defined as
  \begin{equation}\label{6.9}
    \mathcal{R}(\mathbf{b}) = \{\rho(x)x;~x\in \bar{\Omega},~\rho(x)=\min_{1\leq j \leq m}h(x,P_{j},b_{j})\},
  \end{equation}
  where
  \begin{equation*}
    h(x,P_{j},b_{j}) = \frac{( b_{j} - \kappa^{2} x\cdot P_{j}) + \sqrt{(b_{j} - \kappa^{2} x\cdot P_{j})^{2} - (1 - \kappa^{2})( b_{j}^{2} - \kappa^{2}\vert P_{j} \vert^{2} )}}{1 - \kappa^{2}}.
  \end{equation*}
  \label{rem6.2}
\end{remark}

Now we show the existence of the weak solution when $\mu$ equals finite sum of $\delta$-measures.

\begin{theorem}
  Suppose that the assumptions $(B1)$ $-$ $(B5)$ hold. Let $\mu$ be the Borel measure defined on $\bar{D}$ by $\mu = \sum\limits_{i = 1}^{m}g_{i}\delta_{m_{i}}(\omega)$ with $\mu(\bar{D}) = \sum\limits_{j = 1}^{m} g_{j}$, where $\omega \subset \bar{D}$ is a Borel set. Then for any $b_{1}$ satisfies the condition $\kappa \vert P_{1} \vert <b_{1} \leq \kappa \vert P_{1} \vert + r_{0}(1 + \kappa)$, there exists $(b_{2},\ldots,b_{m})$, such that the refracting oval $\mathcal{R} = \{\rho(x)x;~x\in \bar{\Omega}\}$ is the weak solution to the near field refraction problem with loss of energy for the case $-1< \kappa < 0$, where $\rho(x) = \min\limits_{1 \leq j \leq m}h(x,P_{j},b_{j})$.
  \label{thm6.2}
\end{theorem}

Similar to Section \ref{Section 6.1}, in order to prove Theorem \ref{thm6.2}, we also need the following Lemmas \ref{lem6.4} $-$ \ref{lem6.6}.

\begin{lemma}
  Suppose that assumptions $(B1)$ and $(B2)$ hold. Define a set $W := \{(b_{2},\ldots,b_{m});~\kappa \vert P_{j} \vert < b_{j} \leq (\tau - \kappa)\vert P_{j} \vert,~j = 2,\ldots,m\}$, and for any $\mathbf{b} \in W$, we have $\displaystyle\int_{\mathcal{T}_{\mathcal{R}(\mathbf{b})}(P_{j})}f(x)t_{\mathcal{R}(\mathbf{b})}(x)dx \leq g_{i}, 2 \leq i \leq m$, then
  \newline (a) $W \neq \emptyset$;
  \newline (b) If $\mathbf{b} = (b_{2}, \ldots ,b_{m}) \in W$, then for $2 \leq j \leq m$, we have 
  \begin{equation*}
    b_{j} \geq \kappa \vert P_{j} \vert + \dfrac{1 + \kappa}{1 - \kappa}(b_{1} - \kappa \vert P_{1} \vert).
  \end{equation*}
  \label{lem6.4}
\end{lemma}

\begin{proof}
  $(a)$ Let $b_{j} = (\tau - \kappa)\vert P_{j} \vert$ for $j = 2,\ldots,m$, we claim that $\mathcal{R}(\mathbf{b}) = \mathcal{O}(P_{1},b_{1})$, where $\mathcal{O}(P_{1},b_{1}) = \{h(x,P_{1},b_{1});~x\in S^{n-1},x \cdot P_{1} \geq b_{1}\}$.
  
  Indeed, from Lemma \ref{lem3.3}, we have
  \begin{equation*}
    \begin{aligned}
    h(x,P_{1},b_{1}) & \leq \frac{b_{1} - \kappa \vert P_{1} \vert}{1 + \kappa} \leq r_{0} \leq \frac{\tau}{1 - \kappa}dist(0,\bar{D}) \\
     & \leq \frac{\tau \vert P_{j} \vert}{1 - \kappa} = \frac{b_{j} + \kappa \vert P_{j} \vert}{1 - \kappa} \\
     & \leq \frac{b_{j} - \kappa \vert P_{j} \vert}{1 - \kappa} \leq h(x,P_{j},b_{j}),
    \end{aligned}
  \end{equation*}
  where we have used the conditiom $\kappa \vert P_{1} \vert < b_{1} \leq \kappa \vert P_{1} \vert + r_{0}(1 + \kappa)$ in the second inequality. Then for $j \neq 1$, we have $G_{\mathcal{R}(\mathbf{b})}(P_{j}) = 0 \leq g_{i}$. Hence $W \neq \emptyset$.
  
  $(b)$ For $\mathcal{R}(\mathbf{b})$ is a near field refractor, then there exists $x_{0} \in \bar{\Omega}$, such that $\mathcal{O}(P_{1},b_{1})$ supports $\mathcal{R}(\mathbf{b})$ at $\rho(x_{0})x_{0}$, then we have
  \begin{equation*}
    \frac{b_{1} - \kappa \vert P_{1} \vert}{1 - \kappa} \leq \rho(x_{0}) \leq h(x_{0},P_{j},b_{j}) \leq \frac{b_{j} - \kappa \vert P_{j} \vert}{1 + \kappa}.
  \end{equation*}
  Hence $b_{j} \geq \kappa \vert P_{j} \vert + \dfrac{1 + \kappa}{1 - \kappa}(b_{1} - \kappa \vert P_{1} \vert)$.
\end{proof}

The following two lemmas are similar to Lemmas \ref{lem6.2} and \ref{lem6.3}.

\begin{lemma}
  Suppose that $\mathbf{b}_{k} = (b_{2}^{k},\ldots,b_{m}^{k}) \in W$ and $\mathbf{b}_{0} = (b_{2}^{0},\ldots,b_{m}^{0})$, such that $\mathbf{b}_{k} \rightarrow \mathbf{b}_{0}$ as $k \rightarrow \infty$. Let $\mathcal{R}_{k} = \mathcal{R}(\mathbf{b}_{k}) = \{\rho_{k}(x)x;~x \in \bar{\Omega}\}$ and $\mathcal{R}_{0} = \mathcal{R}(\mathbf{b}_{0}) = \{\rho(x)x;~x \in \bar{\Omega}\}$, then $\rho_{k} \rightarrow \rho$ uniformly on $\bar{\Omega}$.
  \label{lem6.5}
\end{lemma}

\begin{proof}
  From the proof of Lemma \ref{lem6.4} $(b)$, we have 
  \begin{equation}\label{6.10}
    \frac{b_{1} - \kappa \vert P_{1} \vert}{1 - \kappa} \leq \rho_{k}(x).
  \end{equation}
  Multiplying  both sides of (\ref{6.10}) by $\dfrac{1 + \kappa}{1 - \kappa} < 1$ and combining it with Lemma \ref{lem3.3}, we have
  \begin{equation*}
    \frac{1 + \kappa}{(1 - \kappa)^{2}}(b_{1} - \kappa \vert P_{1} \vert) \leq \frac{1 + \kappa}{1 - \kappa}\rho_{k}(x) \leq \rho_{k}(x) = \min_{1 \leq j \leq m}h(x,P_{j},b_{j}^{k}) \leq r_{0},
  \end{equation*}
  where $b_{1}^{k} = b_{1}$. Hence $\rho_{k}(x) \rightarrow \rho(x)$ uniformly on $\bar{\Omega}$.
\end{proof}

\begin{lemma}
  Suppose that $\delta \geq \vert P_{j} \vert + \dfrac{1 + \kappa}{1 - \kappa}(b_{1} - \kappa \vert P_{1} \vert)$, then $G_{\mathcal{R}(\mathbf{b})}(\{P_{j}\}) = \displaystyle\int_{\mathcal{T}_{\mathcal{R}(\mathbf{b})}(P_{j})}f(x)t_{\mathcal{R}}(x)dx$ is continuous on the region $R_{\delta} = \{(b_{2},\ldots,b_{m});~b_{j} \geq \delta,j=2,\ldots,m\}$.
  \label{lem6.6}
\end{lemma}

The proof of Lemma \ref{lem6.6} is similar to the proof of Lemma \ref{lem6.3}, for any $x \in \bar{\Omega}$, here we have
  \begin{equation*}
   0< \frac{1 + \kappa}{(1 - \kappa)^{2}}(b_{1} - \kappa \vert P_{1} \vert) \leq \rho_{k}(x) \leq r_{0}.
  \end{equation*}

Based on the above lemmas, now we prove the existence of the weak solution for the case $-1< \kappa < 0$ when $\mu$ equals finite sum of $\delta$-measures.

\begin{proof}[Proof of Theorem \ref{thm6.2}]
  In order to prove that the refracting oval $\mathcal{R} = \{\rho(x)x;~x\in \bar{\Omega}\}$ is the weak solution to the near field refraction problem with loss of energy for the case $-1< \kappa < 0$, we need to verify the following two conditions:
  
  (a) For any Borel set $F \subset \bar{D}$, there holds $\displaystyle\int_{\mathcal{T}_{\mathcal{R}}(F)}f(x)t_{\mathcal{R}}(x)dx \geq \mu(F)$;
  
  (b) For each Borel set with $P_{1} \notin F$, there holds $\displaystyle\int_{\mathcal{T}_{\mathcal{R}}(F)}f(x)t_{\mathcal{R}}(x)dx = \mu(F) = \sum\limits_{j=2}^{m}g_{j}$.

  We consider the mapping $d:W \rightarrow \mathbb{R}:~(b_{2},\ldots,b_{m}) \mapsto b_{2} + \ldots + b_{m}$. From Lemmas \ref{lem6.4} $-$ \ref{lem6.6}, we know that $W$ is a compact set. Hence $d$ can attain its minimum at some points in $W$. Suppose that $d$ attains its maximum at $(a_{2},\ldots,a_{m})$, we claim that $\mathcal{R}(\mathbf{a})$ is the weak solution to the near field refraction problem, where $\mathbf{a} = (b_{1},a_{2},\ldots,a_{m})$.
  
  We first prove that $\displaystyle\int_{\mathcal{T}_{\mathcal{R}(\mathbf{a})}(P_{j})}f(x)t_{\mathcal{R}}(x)dx = g_{j}$ for $j = 2,\ldots,m$.
  
  Indeed, without loss of generality, we assume for contradiction that $\displaystyle\int_{\mathcal{T}_{\mathcal{R}(\mathbf{a})}(P_{2})}f(x)t_{\mathcal{R}}(x)dx < g_{2}$. Let $\varepsilon > 0$ and define $\bar{\mathbf{a}} = (b_{1},a_{2}-\varepsilon,\ldots,a_{m})$, the corresponding refracting oval is denoted as $\mathcal{R}(\bar{\mathbf{a}})$. From the continuity of $G_{\mathcal{R}}$, there holds $G_{\mathcal{R}(\bar{\mathbf{a}})}(P_{2}) = \displaystyle\int_{\mathcal{T}_{\mathcal{R}(\mathbf{a})}(P_{2})}f(x)t_{\mathcal{R}(\mathbf{a})}(x)dx < g_{2}$ for $\varepsilon$ is sufficiently small, and for $j \neq 1,2$, we have $\mathcal{T}_{\mathcal{R}(\bar{\mathbf{a}})}(P_{j}) \subset \mathcal{T}_{\mathcal{R}(\mathbf{a})}(P_{j})$ almost everywhere, hence $\bar{\mathbf{a}} \in W$. This is a contradiction with $(a_{2},\ldots,a_{m})$ is the minimum of $d$ in $W$, so we have $\displaystyle\int_{\mathcal{T}_{\mathcal{R}(\mathbf{a})}(P_{j})}f(x)t_{\mathcal{R}}(x)dx = g_{j}$ for $j = 2,\ldots,m$.
  
  To prove that $\displaystyle\int_{\mathcal{T}_{\mathcal{R}(\mathbf{a})}(P_{1})}f(x)t_{\mathcal{R}}(x)dx > g_{1}$, we proceed exactly as in the proof of Theorem \ref{thm6.1}.
\end{proof}

Based on Theorem \ref{thm6.2}, now we can prove the existence of the weak solution to the near field refraction problem with loss of energy for the case $-1< \kappa < 0$.

\begin{proof}[Proof of Theorem \ref{thm1.2}]
  In order to prove the existence of the weak solution for the case $-1 < \kappa < 0$, we need to show that given $P_{0} \in supp(\mu)$, then for any Borel set $F \subset D$ with $P_{0} \notin F$, there exists a refractor $\mathcal{R}$, such that $\displaystyle\int_{\mathcal{T}_{\mathcal{R}}(F)}f(x)t_{\mathcal{R}}(x)dx \geq \mu(F)$, where $\mu$ is the Radon measure defined on $\bar{D}$.
  
  The process of proof follows from the analogue of Theorem \ref{thm1.1}, which now follows from Theorem \ref{thm6.2} rather than Theorem \ref{thm6.1}.
\end{proof}

\subsection{Case $\kappa = -1$}\label{Section 6.3}

A very interesting phenomenon in negative refraction is the case when $\kappa = -1$. Notice that in positive refraction, if $\kappa = 1$, then the refraction phenomenon do not occur. However, in negative refraction, the refraction phenomenon still occurs when $\kappa = -1$. Specifically, considering (\ref{4.8}), we have $R_{\parallel} = R_{\bot} = 0$ for this case. Thus all the energy is transmitted and nothing is reflected internally. This special refraction phenomenon can be described by Fig \ref{fig5}. In this subsection, we briefly introduce this special phenomenon.

\begin{figure}[h]
  \centering
  \includegraphics[width=8.3cm]{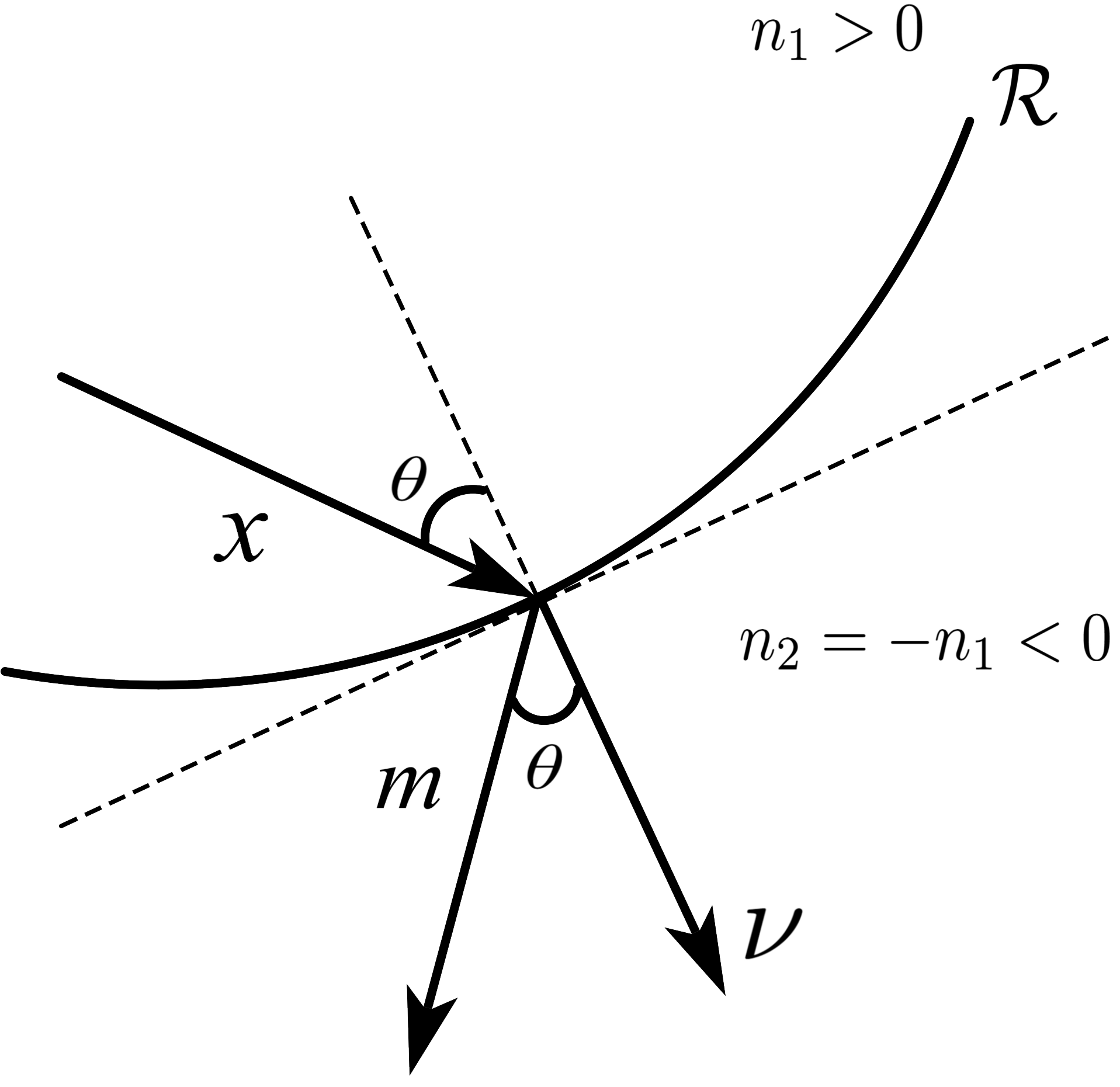}
\caption{Sketch of the refraction with loss of energy when $\kappa = -1$.}\label{fig5}
\end{figure}

From \cite{GS15}, we know that in the case $\kappa = -1$, the refracting surface is a semi-hyperboloid given by
$$\mathcal{E}(P,b) := \{h(x,P,b)x;~x \in S^{n-1},~ x \cdot P > b\},$$
where $h(x,P,b) = \dfrac{b^{2} - \vert P \vert^{2}}{2(b - x \cdot P)}$ with $\vert b \vert \leq \vert P \vert$.

Now we can define the notion of the refractor for the case $\kappa = -1$.

\begin{definition}
  Let $\mathcal{R} = \{\rho(x)x; x \in \bar{\Omega}\} \subset \mathcal{C}_{r_{0}}$ be a surface. Then $\mathcal{R}$ is called a near field refractor for the case $\kappa = -1$ if for any $\rho(x_{0})x_{0} \in \mathcal{R}$, there exist $P \in \bar{D}$ and a number $b$ satisfies $\kappa \vert P \vert < b < \vert P \vert$, such that the refracting oval $\mathcal{E}(P,b)$ supports $\mathcal{R}$ at $\rho(x_{0})x_{0}$. Namely, for any $ x \in \bar{\Omega}$, we have $\rho(x) \leq h(x,P,b)$ and $\rho(x_{0}) = h(x_{0},P,b)$. 
  \label{def6.1}
\end{definition}

The following Figure \ref{fig6} shows the refracting semi-hyperboloid which refracts all ray emitted from the source $O$ to a specific point $P$ for the case $\kappa = -1$.

\begin{figure}[h]
  \centering
  \includegraphics[width=8.3cm]{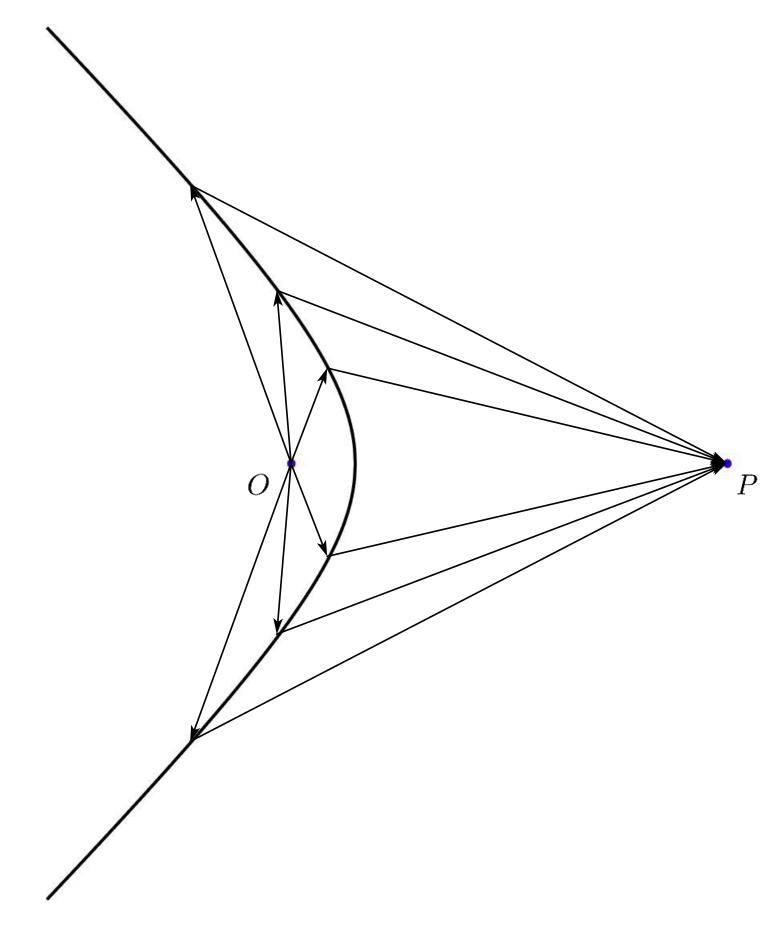}
\caption{Refracting semi-hyperboloid when $\kappa = -1$, where $O$ and $P$ are the focus of hyperboloid.}\label{fig6}
\end{figure}

According to the analysis in Section \ref{Section 3}, we know that the Lipschitz continuity of $\rho$ is crucial for proving the existence of the weak solution. Therefore, we show the Lipschitz continuity of $\rho$ here briefly for the case $\kappa = -1$.

\begin{lemma}
  Suppose that $\mathcal{R} = \{\rho(x)x;~x\in \bar{\Omega}\}$ is a near field refractor for the case $\kappa = -1$, then $\rho$ is Lipschitz continuous with the Lipschitz constant only depending on $b$ and $\vert P \vert$.
  \label{lem6.7}
\end{lemma}

\begin{proof}
  Suppose that $\mathcal{E}(P,b)$ supports $\mathcal{R}$ at $\rho(x)x$, then
  $$\rho(x) = h(x,P,b) \quad \text{and} \quad \rho(y) \leq h(y,P,b)~\text{for all} ~y \in \bar{\Omega}.$$
  So we have
  \begin{equation*}
    \begin{aligned}
    \vert \rho(y) - \rho(x) \vert & \leq \vert h(y,P,b) - h(x,P,b) \vert \leq \frac{1}{2} \left\vert \frac{b^{2} - \vert P \vert^{2}}{b - y \cdot P} - \frac{b^{2} - \vert P \vert^{2}}{b - x \cdot P} \right\vert \\
    & \leq \frac{1}{2} \left\vert \frac{(b + \vert P \vert)(b - \vert P \vert)\vert P \vert}{(b - \vert P \vert)^{2}} \right\vert \Vert x - y \Vert \\
    & \leq \frac{1}{2} \left\vert \frac{(b + \vert P \vert)\vert P \vert}{(b - \vert P \vert)}\right\vert \Vert x - y \Vert.
    \end{aligned}
  \end{equation*}
  Exchanging the roles of $x$ and $y$, then we can get $\vert \rho(x) - \rho(y) \vert \leq L \Vert x - y \Vert$ for $L = \dfrac{1}{2} \left\vert \dfrac{(b + \vert P \vert)\vert P \vert}{(b - \vert P \vert)}\right\vert > 0$. Hence $\rho$ is Lipschitz continuous on $\bar{\Omega}$.
  
\end{proof}

The other properties of the refractor for the case $\kappa = -1$ can be proved by using analogous methods in Section \ref{Section 3.3}. Based on these properties, the existence of the weak solution for the case $\kappa = -1$ can be proved. The proof of existence of the weak solution for this case is much simpler than the cases $\kappa < -1$ and $-1 < \kappa < 0$. Therefore, we shall omit the detailed description.

\section{Conclusion}\label{Section 7}

\sloppy{}

In this paper, we studied the near field refraction problem with loss of energy in negative refractive index material. The Snell law in vector form was first reviewed, and the physical constraints governing refraction were discussed. The refractor for the cases $\kappa < -1$ and $-1 < \kappa < 0$ was then defined, and some important properties were analyzed. As a key tool for describing the loss of energy, Fresnel coefficients and their properties were discussed subsequently. Finally, we defined the weak solution to the near field refraction problem with loss of energy in negative refractive index material and proved the existence of the weak solution for the case $\kappa < -1$ and $-1 < \kappa < 0$. Additionally, the critical case $\kappa = -1$ was briefly discussed at the end. The process of the proof of the existence of the weak solution to the near field refraction with loss of energy is different from the proof of the existence of the weak solution to the near field refraction without loss of energy in \cite{St17}, especially for the case $\kappa < -1$. In \cite{St17}, since the loss of energy is not considered, the energy conservation condition
$$\int_{\Omega} f(x) dx = \mu(\bar{D})$$
is satisfied. So in order to prove the existence of weak solution, it is only necessary to verify the condition of Theorem 2.9 in \cite{GH14}. Because of the loss of energy, we need to use Minkowski method to prove the existence of the weak solution for each cases. Besides, when proving the existence of the weak solution for the case $\kappa < -1$ when $\mu$ equals finite sum of $\delta$-measures, we choose a different $\alpha$ from \cite{St17} to ensure that $\mathcal{R}(\mathbf{b}) = \Gamma(P_{1},b_{1})$. Meanwhile, since we consider the near field refraction problem in negative refractive index material, we use a different method from that in  Chapter 7 of \cite{St16} for the scaling of the inequality when proving the boundedness of some parameters.

This paper used the Minkowski method to solve the near field refraction problem with loss of energy in negative refractive index material, which is a remaining problem in \cite{St16}. Minkowski method is an iterative approach for solving refraction and reflection problems in geometric optics, which is effective in solving refraction problem with loss of energy. In \cite{LW21}, Liu and Wang formulate the near field refraction problem without loss of energy into a nonlinear optimization problem. However, can the near field refraction problem with loss of energy be formulated as a nonlinear optimization problem is still an open problem. Additionally, the existence of the weak solution to the near field refraction problem without loss of energy can be proved by using generated Jacobian equation \cite{Tr14}. Can generated Jacobian equation be used to prove the existence of the weak solution to the near field refraction problem with loss of energy is still an open problem as well.

\vspace{3mm}

\noindent {\bf Conflict of Interest}  { The authors declare no conflict of interest.}

\vspace{3mm}

\noindent {\bf Acknowledgments}  { This work was supported by the National Natural Science Foundation of China (No. 12271093), the Jiangsu Provincial Scientific Research Center of Applied Mathematics (No. BK20233002) and the Start-up Research Fund of Southeast University (No. 4007012503).}

\end{document}